\documentclass[a4paper,12pt,twoside]{amsart}
\usepackage{amsthm,amssymb,amsmath,amscd,url}
\newtheorem{thm}{Theorem}[section]
\newtheorem{conjecture}[thm]{Conjecture}

\newtheorem{lem}[thm]{Lemma}
\newtheorem{prop}[thm]{Proposition}

\newtheorem{rem}[thm]{Remark}
\newtheorem{notation}[thm]{Notation}

\newtheorem{defn}[thm]{Definition}

\newcommand{\Aset}{\mathbb{A}}

\newcommand{\Zset}{\mathbb{Z}}
\newcommand{\Rset}{\mathbb{R}}
\newcommand{\Cset}{\mathbb{C}}

\def\red{{\rm red}}

\def\orb{{\mathcal O}}
\newcommand{\Qset}{\mathbb{Q}}

\newcommand{\bB}{\mathbf{B}}

\newcommand{\bc}{\mathbf{c}}

\newcommand{\ab}{{\rm ab}}

\newcommand{\tempere}{{\rm temp}}
\newcommand{\ttop}{{\rm top}}
\newcommand{\reel}{{\rm rc}}

\def\Sc{{\mathbf {Sc}}}

\def\fp{{\mathfrak p}}

\def\PPhi{{\Phi}}

\def\rG{{{\rm G}}}

\def\hchi{{{\widehat\chi}}}

\def\cB{{\mathcal B}}
\def\cC{{\mathcal C}}

\def\Cent{{\rm C}}
\def\Nor{{\rm N}}

\def\ie{{\it i.e.,\,}}
\def\cE{{\mathcal E}}
\def\cI{{\mathcal I}}

\def\integers{{\mathfrak o}}

\def\im{{\rm im\,}}

\def\red{{\rm r}}

\def\unip{{\rm u}}
\def\RBC{{\rm BC}}
\def\H{{\rm H}}
\def\HP{{\rm HP}}
\def\SL{{\rm SL}}

\def\SO{{\rm SO}}
\def\GL{{\rm GL}}

\def\Aut{{\rm Aut}}
\def\Hom{{\rm Hom}}

\def\conj{{\rm conj}}

\def\Ind{{\rm Ind}}
\def\Lie{{\rm Lie}}

\def\der{{\rm der}}

\def\Nor{{\rm N}}
\def\orb{{\mathcal O}}
\def\Irr{{\mathbf {Irr}}}

\def\Lie{{\rm Lie}}

\def\Prim{{\rm Prim}}

\def\cA{{\mathcal A}}
\def\cR{{\mathcal R}}
\def\cG{{\mathcal G}}

\def\cH{{\mathcal H}}
\def\cM{{\mathcal M}}
\def\cO{{\mathcal O}}

\def\cT{{\mathcal T}}
\def\cU{{\mathcal U}}

\def\cW{{\mathcal W}}

\def\Irrt{{{\rm Irr}^{\rm t}}}

\def\fb{{\mathfrak b}}
\def\fc{{\mathfrak c}}

\def\fs{{\mathfrak s}}
\def\ft{{\mathfrak t}}

\def\fB{{\mathfrak B}}

\def\fii{{\mathfrak i}}

\def\fn{{\mathfrak n}}
\def\fG{{\mathfrak G}}

\def\fP{{\mathfrak P}}
\def\fQ{{\mathfrak Q}}
\def\integers{{\mathfrak o}}

\def\bZ{{\mathbb Z}}

\def\fB{{\mathfrak B}}

\def\fp{{\mathfrak p}}

\def\St{{\rm St}}

\def\tH{{\widetilde H}}
\def\th{{\tilde h}}
\def\ttt{{\tilde t}}

\def\tM{{\widetilde M}}
\def\tu{{\tilde u}}

\def\q{{/\!/}}

\def\ab{{\rm ab}}
\def\der{{\rm der}}
\def\sgn{{\rm sgn}}
\def\pphi{{\phi}}
\def\PPhi_F{{\rm Frob}}
\def\Frob{{\rm Frob}}
\def\Flag{{\rm Flag}}

\keywords{Extended quotients, representations of $p$-adic groups, principal
series}

\begin{document}
\title[Extended quotients in the principal series]{Extended quotients in the principal series of reductive $p$-adic groups}
\author{Anne-Marie Aubert, Paul Baum and Roger Plymen}
\address{A.-M.~Aubert: Institut de Math\'ematiques de Jussieu
(U.M.R. 7586 du C.N.R.S.), Universit\'e Pierre et Marie Curie, 
4 place Jussieu, 75005 Paris, France}
\email{aubert@math.jussieu.fr}
\address{P.~Baum: Pennsylvania State University, Mathematics Department, 
University Park, PA 16802, USA}
\email{baum@math.psu.edu}
\address{R.~Plymen: School of Mathematics, Manchester University, Manchester M13 9PL, England \emph{and} School of Mathematics, Southampton University,
SO17 1BJ, England}
\email{plymen@manchester.ac.uk, r.j.plymen@soton.ac.uk}
\date{\today}
\maketitle
\thanks{}
\medskip

\begin{abstract} The geometric conjecture developed by the authors in \cite{ABP1, ABP2, ABP3, ABP4} applies to the smooth dual $\Irr(\cG)$ of any reductive $p$-adic group $\cG$.  It predicts a definite geometric structure -- the structure of an extended quotient -- for each component $\Irr(\cG)^{\fs}$ in the Bernstein decomposition of $\Irr(\cG)$. 

In this article, we prove the geometric conjecture for the principal series in any split connected reductive $p$-adic group $\cG$.   The proof proceeds via Springer parameters and Langlands parameters.   As a consequence of this approach, we establish strong links with the local Langlands correspondence.

One important  feature of our approach is the emphasis on 
%unipotent classes in the Langlands dual $\cG^{\vee}$. 
two-sided cells $\bc$ in  extended affine Weyl groups.
\end{abstract}

\section{\ Introduction}    In a series of papers \cite{ABP1, ABP2, ABP3, ABP4} we have proposed a geometric conjecture about the smooth dual of any reductive $p$-adic group $\cG$.   According to the conjecture, if $\Irr(\cG)^{\fs}$ is any Bernstein component \cite{Renard} of the smooth dual $\Irr(\cG)$, then $\Irr(\cG)^{\fs}$ is in bijection with the extended quotient $T^{\fs}\q W^{\fs}$. Here $T^{\fs}$ is the complex torus and $W^{\fs}$ is the finite group (acting on $T^{\fs}$) which Bernstein assigns to the point $\fs$ of the Bernstein spectrum of $\cG$. The point of our conjecture is that $\Irr(\cG)^{\fs}$ can, in practice, be quite difficult to calculate --- but the extended quotient $T^{\fs}\q W^{\fs}$ is always easy to calculate.

More precisely, our conjecture can be stated at four levels:
\begin{itemize}
\item $K$-theory
\item Periodic cyclic homology
\item Geometric equivalence of finite type algebras
\item Representation theory
\end{itemize}

In this paper we shall consider the conjecture at the level of representation theory.   An evident question is the compatibility of our extended quotient with already known parametrizations of certain Bernstein components.   In this paper we shall establish the compatibility for the Kazhdan-Lusztig-Reeder parametrization \cite{R} of Bernstein components in the principal series.  Thus if $\cG$ is any split
connected reductive $p$-adic group with
maximal split torus $\mathcal{T}$, 
and $\Irr(\cG)^{\fs}$ is a Bernstein 
component in the principal series of $\cG$ (with the assumption, when 
$\fs\ne[\cT,1]_\cG$, that $\cG$ has connected centre and $p$ is not too small),
we construct a canonical bijection 
\[
\mu^\fs\colon(T^{\fs}\q W^{\fs})_2 \to \Irr(\cG)^{\fs}
\]
where $\Irr(\cG)^{\fs}$ is parametrized by Kazhdan-Lusztig-Reeder parameters, and $(T^{\fs}\q W^{\fs})_2$ is the extended quotient of the second kind (as explained in \S3). 

We now supply some more details.   For each $\alpha\in\Cset^\times$, we  construct a commutative diagram (see
section~9):
\begin{equation}\label{CD0}
\begin{CD}
(T^\fs\q W^{\fs})_2 @> \mu^{\fs} >>  \Irr(\cG)^{\fs}\\
@ V \pi_{\alpha} VV        @VV i_{\alpha} V\\
T^\fs/W^{\fs} @=  T^\fs/W^{\fs}
\end{CD}
\end{equation}
such that
\begin{equation}\label{root}
\pi_{\sqrt q} = \mathbf{Sc} \circ \mu^{\fs}.
\end{equation}
Here,  $q$ is the cardinality of the residue field of the underlying 
$p$-field $F$, and  
 $\mathbf{Sc}$ denotes the \emph{cuspidal support} map, also called
the \emph{infinitesimal character}, from $\Irr(\cG)$ to the set of all 
$\cG$-conjugacy classes of cuspidal pairs.
The second vertical map $i_{\alpha}$ in the diagram is defined in terms of
Kazhdan-Lusztig-Reeder parameters (see Eqn.~(\ref{eqn:iq})) while
the first vertical map $\pi_{\alpha}$ is built from 
correcting cocharacters, see \S11.

We prove that the number of points in the fibre of $\pi_{\sqrt q}$ equals the 
number of inequivalent irreducible constituents of 
$\Ind_{\mathcal{B}}^{\mathcal{G}}(\chi)$, where
$\cB\supset\cT$ is the standard Borel subgroup in $\cG$, and
$\fs=[\cT,\chi]_\cG$.
This confirms our geometric conjecture, as stated in \cite{ABP3}, for the
principal series of $\cG$.

This conjecture has been proved independently by Maarten Solleveld in
\cite{Sol} in a completely different way: we note that
Solleveld's theorem applies to situations more general  than the principal
series of a reductive $p$-adic group.
 
Our proof provides, on the case of the principal series, 
a result more precise than \cite[Theorem~2]{Sol} in the sense that we 
obtain (see Eqn.~(\ref{cell})) a cell-decomposition 
 \begin{align}\label{cells}
 (T^\fs\q W^{\fs})_2 = \bigsqcup (T^\fs\q W^{\fs})_2^{\bc}
 \end{align} 
where $\bc$ runs over the two-sided cells of the extended affine Weyl
group $X(T^\fs)\rtimes W^\fs$. Here $X(T^\fs)$ is the group of characters of $T^\fs$.   

We will refer to  $(T^\fs\q W^{\fs})_2^{\bc}$ as the $\bc$-component of $(T^\fs\q W^{\fs})_2$.
We prove, as conjectured in \cite[p.~87]{ABP4}, that
the following properties hold:
\begin{enumerate}
\item
the ordinary quotient $T^\fs/W^\fs$ is contained in the $\bc_0$-component of  $(T^\fs\q W^{\fs})_2$
where $\bc_0$ is the lowest two-sided cell 
\item
each correcting cocharacter is attached to a two-sided cell $\bc$.
\end{enumerate}
\smallskip

The cell decomposition in Eqn.~(\ref{cells}) is a key ingredient of our
characterization of the intersections of the $L$-packets with $\Irr(\cG)^\fs$ 
in the principal series case.
We prove in Theorem~\ref{thm:Lpackets} that
two points $(t, \tau)$ and $(t^{\prime}, \tau^{\prime})$ in
$(T^\fs\q W^\fs)_2$ have their images by the bijection $\mu^\fs$ in the
same $L$-packet if and only if $(t, \tau)$ and $(t^{\prime}, \tau^{\prime})$ 
belong to the same $\bc$-component in Eqn.(\ref{cells})  and $\pi_{\alpha}(t,\tau) =
\pi_{\alpha}(t^{\prime},\tau^{\prime})$ for each $\alpha\in\Cset^\times$.

\smallskip
 Solleveld writes, in the Introduction of \cite{Sol}, that that
he hopes that his Theorem 2 will be useful in the local Langlands
program and relate the semisimple element $t$ to the image of the
geometric Frobenius by the Langlands parameter. This is exactly what we achieve (by an independent route) in our 
definition of the $L$-parameter $\Phi$, see Eqn.(\ref{Phi}).  Solleveld also writes that ``The rest of such
a Langlands parameter $\Phi$ is completely beyond affine Hecke algebras,
it will have to depend on number theoretic properties of the Bernstein components $\fs$''. Our results show that for principal series representations
at least this is not the case, see once again Eqn.(\ref{Phi}).  The cell decomposition is governed by
data attached to the affine Hecke algebra and is intimately related to
the $L$-packets, see  \S12.
\smallskip

Note that, in Eqn.(\ref{root}), we have set
\begin{equation}\label{alpha}
\alpha = \sqrt q.
\end{equation}
 In the commutative diagram (\ref{CD0}), all values of $\alpha \in \Cset^{\times}$ are allowed.  In particular, we may set $\alpha = 1$.
 In view of Eqn.(\ref{alpha}), the field $\mathbb{F}_1$ with one element appears, at the very least, in the background of this article.   We hope to develop this point of view elsewhere. %The interpretation of this diagram  for these values of  $\alpha$ is not, at present, entirely clear to us;  i
  %Assume, for the moment, that $C_H (t)$ is connected. The map $\mu$ in Theorem 4.1 sends $(t, \tau)$    to $(\Phi, \rho)$. We note that  $t = \Phi
 %(\Frob, T_1) = \Phi(\Frob, 1)$. The map $\mu$ determines the map 
 %\[ (t, \tau) \mapsto (t, \Phi(1, u_0), \rho) \] which, in turn, determines the map  \[ \tau \mapsto  (\exp(x), \rho) \] which is the Springer correspondence 5for the Weyl group $W_H(t)$.
 
 \medskip 

\tableofcontents

\section{\ Extended quotients}   Let $X$ be an affine algebraic variety over $\Cset$, and let   $\mathcal{O}(X)$ denote the coordinate algebra of $X$.  
Suppose t¤hat $W$ is a finite group acting on $X$ as automorphisms of the affine variety $X$.   The quotient $X/W$ is obtained by collapsing each orbit to a point and is again an affine variety.

For $x \in X$, $W_x$ denotes the stabilizer group of $x$:
\[
W_x: = \{w \in W : wx = x\}.
\]
Now $c(W_x)$ is the set of conjugacy classes of $W_x$.  The extended quotient, denoted $X\q W$, is constructed by replacing each orbit with $c(W_x)$ where $x$ can be any point in the orbit. This construction is done as follows.

First, set
\[
\widetilde{X}: = \{(w,x) \in W \times X : wx = x\}.
\]
Then $\widetilde{X} \subset W \times X$ and is an alegbraic sub-variety of the affine variety $W \times X$. Then $W$ acts on $\widetilde{X}$:
\[
W \times \widetilde{X} \to \widetilde{X}, \quad \quad \alpha(w,x) = (\alpha w \alpha^{-1}, \alpha x)
\]
with $(w,x) \in \widetilde{X}, \alpha \in W$. Then we define
\[
X\q W: = \widetilde{X}/W.
\]
Hence the extended quotient $X\q W$ is the ordinary quotient for the action of $W$ on $\widetilde{X}$. In a straightforward way, $X\q W$ is an affine variety.
The map
\[
\widetilde{X} \to X, \quad \quad (w,x) \mapsto x
\]
is equivariant and thus passes to quotient spaces as a morphism of algebraic varieties $X\q W \to X/W$ which will be referred to as the projection of the extended quotient on the quotient.

\begin{rem}
{\rm  Let $\cO(X) \rtimes W$ be the crossed product algebra for the action of $W$ on $\cO(X)$.   Denote the periodic cyclic homology of this crossed product algebra by $\HP_j(\cO(X)\rtimes W)$ with $j = 0,1$.   There is a canonical isomorphism \cite{B} of $\Cset$-vector spaces
\begin{align}\label{cyclic}
\HP_j(\cO(X) \rtimes W) \simeq \bigoplus_l \H^{j+2l}(X\q W); \Cset)
\end{align}
where, as usual, $\H^*( X\q W; \Cset)$ is the cohomology of $X\q W$ (using its analytic topology).  The right-hand-side of Eqn.(\ref{cyclic}) is also known as the
\emph{orbifold cohomology} $H^*_{orb}(X/W ; \Cset)$ of the quotient $X/W$, see \cite{Bara}.
}
\end{rem}

The ABP (Aubert-Baum-Plymen) conjecture at the level of periodic cyclic homology is:

\begin{conjecture} 
Let $\cG$ be a reductive $p$-adic group and let $\fs$ be a point in the Bernstein spectrum of $\cG$. there there are canonical isomorphisms of $\Cset$ vector spaces
\[
\HP_j(\cI^{\fs}) \simeq \bigoplus_l \H^{j+2l}(T^{\fs}\q W^{\fs} ; \Cset)
\]
with $j = 0,1$ where $\cI^{\fs} \subset \cH(\cG)$ is the Bernstein ideal in $\cH(\cG)$ labelled by $\fs$, and $T^{\fs}, \,W^{\fs}$ are as above.
\end{conjecture}

\section{\ Extended quotients of the second kind}

As in the previous section, let $X$ be a complex affine variety with a finite group $W$ acting as automorphisms of $X$.   For each $x \in X$, $\Irr(W_x)$ is the set of (equivalence classes of) irreducible representations of the isotropy group $W_x$.   Now $\Irr(W_x)$ has the same number of elements as $c(W_x)$ but, in general, there is no canonical bijection between $\Irr(W_x)$ and $c(W_x)$.   The extended quotient of the second kind, denoted $(X\q W)_2$, is obtained by replacing each orbit by $\Irr(W_x)$, where $x$ is any point in the orbit.    The construction of $(X\q W)_2$ is done as follows.   Define
\[
\widetilde{X}_2: = \{(x,\tau) : x \in X, \tau \in \Irr(W_x)\}.
\]
Then $W$ acts on $\widetilde{X}_2$ by
\[
w(x,\tau) = (wx, w_*\tau)
\]
where $w_*\tau$ is the push-forward by $w$ of $\tau$ from $\Irr(W_x)$ to $\Irr(W_{wx})$. Then 
\[
(X\q W)_2: = \widetilde{X}_2/W
\]
the quotient of $\widetilde{X}_2$ by $W$.  Now
\[
\widetilde{X}_2 \to X, \quad \quad (x, \tau) \mapsto x
\]
is $W$-equivariant and descends to a map of quotient spaces
\[
(X\q W)_2 \to X/W.
\]
This is the projection of the extended quotient of the second kind on the quotient $X/W$. 

There is a non-canonical bijection $\nu \colon X\q W \to (X\q W)_2$ with commutativity in the diagram
\begin{equation} \label{eqn:CDnu}
\begin{CD}
X\q W @>\nu >>(X\q W)_2\\
@VVV               @VVV\\
X/W @> > id > X/W
\end{CD}
\end{equation}
To construct the bijection $\nu$, first let $H_1, \ldots, H_r$ be subgroups of $W$ such that 

(1) For each $j = 1, 2, \ldots, r$, there exists $x \in X$ with $W_x = H_j$

(2) Any $W_x$ is conjugate within $W$ to one and only one $H_j$. 

Now for each $j = 1,2,\ldots, r$ choose a bijection
\[
c(H_j) \to \Irr(H_j).
\]
We shall refer to such a system of bijections as a $c$-$\Irr$ system.   For $x \in X$, $[x]$ denotes the image of $x$ under the quotient map 
$X \to X/W$.   In $X\q W$ the pre-image of $[x]$ with respect to the projection $X\q W \to X/W$ identifies canonically with one and only one $c(H_j)$. 
The pre-image of $[x]$ with respect to the projection $(X\q W)_2 \to X/W$ identifies canonically with one and only one $\Irr(H_j)$.  Then $\nu : X\q W \to
(X\q W)_2$ maps the first pre-image to the second pre-image via the chosen bijection $c(H_j) \to \Irr(H_j)$. 

\begin{lem} \label{lem:csyst}
The bijection 
\[
X\q W \to (X\q W)_2
\]
obtained from a $c$-$\Irr$ system is continuous when $X\q W$ has the Zariski topology and $(X\q W)_2$ has the Jacobson topology.
\end{lem}

\begin{rem}  We have a canonical bijection
\[
\Irr(\cO(X)\rtimes W) \to (X\q W)_2.
\]
The advantage of $X\q W$ over $(X\q W)_2$ is that $X\q W$ is an affine variety, but $(X\q W)_2$ is not (in any canonical way) an algebraic variety. 
\end{rem}

\section{\ Twisted extended quotients of the second kind}  \label{sec: teq}
Let $J$ be a finite group and let $\alpha \in \H^2(J; \Cset^{\times})$.  Consider all maps
$\tau \colon J \to \GL(V)$  such that there exists a  
$\Cset^{\times}$-valued $2$-cocycle $c$ on $J$ with 
\[
\tau(j_1) \circ \tau(j_2)  = c(j_1, j_2)\,\tau(j_1j_2),\quad \quad [c] = \alpha
\]
 and $\tau$ is irreducible, where $[c]$ is the class of $c$ in 
$\H^2(J; \Cset^{\times})$.   

For such a map $\tau$ let $f \colon J \to \Cset^{\times}$ be any map.
 Now consider the map $(f\tau)(j): =  f(j)\tau(j)$. This is again such a map 
$\tau$ with cocycle   $c\cdot f$ defined by
\[(c\cdot f)(j_1,j_2):=c(j_1,j_2)\, \frac{f(j_1)f(j_2)}{f(j_1j_2)}.\]  
Given two such maps $\tau_1, \tau_2$, an \emph{isomorphism} is an intertwining 
operator $V_1 \to V_2$.   
Note that $\tau_1 \simeq \tau_2 \Rightarrow c_1 = c_2$. Given
$\tau_1$ and $\tau_2$,  we
define $\tau_1$ to be \emph{equivalent} to $\tau_2$
if and only if there exists $f \colon J \to \Cset^{\times}$  with
$\tau_1$ isomorphic to $f \tau_2$. The set of equivalence classes of maps 
$\tau$ is denoted $\Irr^{\alpha}(J)$. 

\smallskip

Now let $\Gamma$ be a finite group with a given action on a set $X$. Now let 
$\square$ be a given function which assigns to each $x \in X$ an element 
$\square(x) \in \H^2(\Gamma_x;\Cset^{\times})$ where $\Gamma_x = \{\gamma \in \Gamma: \gamma x = x\}$.  The function $\square$ is required to satisfy the condition
\[
\square(\gamma x) = \gamma_*\square(x), \quad \quad \forall (\gamma,x) \in \Gamma \times X
\]
where $\gamma_*\colon \Gamma_x \to \Gamma_{\gamma x}, \; \alpha \mapsto \gamma \alpha \gamma^{-1}$.   Now define
\[
\widetilde{X}_2^{\square} = \{(x,\tau) : \tau \in \Irr^{\square(x)}(\Gamma_x)\}.
\]
We have a map $\Gamma \times \widetilde{X}_2^{\square} \to \widetilde{X}_2^{\square}$ and we form the \emph{twisted extended quotient of the second kind}
\[
(X\q \Gamma)_2^{\square}: = \widetilde{X}_2^{\square}/\Gamma.
\]

We will apply this construction in the following two special cases.  

\smallskip

{\bf 1.} Given two finite groups $\Gamma_1$, $\Gamma$ and a group homomorphism
$\Gamma \to \Aut(\Gamma_1)$, we can form the semidirect product  $\Gamma_1 \rtimes\Gamma$.  Let $X = \Irr \, \Gamma_1$.   Now $\Gamma$ acts on $\Irr\, \Gamma_1$ and we get $\square$ for free.  Given $x \in \Irr\, \Gamma_1$ choose an irreducible representation  $\phi: \Gamma_1 \to \GL(V)$ whose isomorphism class is $x$. 
For each $\gamma \in \Gamma_x$ consider $\phi$ twisted by $\gamma$ \ie consider $\phi^{\gamma}: \gamma_1 \mapsto \phi(\gamma \gamma_1 \gamma^{-1})$.
Since $\gamma \in \Gamma_x$, $\phi^{\gamma}$ is equivalent to $\phi$ \ie there exists an intertwining operator $T_{\gamma} : \phi \simeq \phi^{\gamma}$. For this operator we have
\[
T_{\gamma} \circ T_{\gamma'} = c(\gamma,\gamma')T_{\gamma \gamma'}, \quad \quad \gamma, \gamma' \in \Gamma_x
\]
and $\square(x)$ is then the class in $\H^2(\Gamma_x;\Cset^{\times})$ of $c$.   

This leads to a new formulation  of a classical theorem of Clifford.
\begin{lem} \label{lem:Clifford} We have a canonical bijection
\[
\Irr(\Gamma_1 \rtimes \Gamma) \simeq (\Irr \, \Gamma_1 \q \Gamma)_2^{\square}.
\]
\end{lem}
\begin{proof}  The proof proceeds by comparing our construction with the classical theory of Clifford; for an exposition of Clifford theory, see \cite{RamRam}. 
\end{proof}

\smallskip

{\bf 2.} Given a $\Cset$-algebra $R$, a finite group $\Gamma$ and a group
homomorphism
$\Gamma \to \Aut(R)$, we can form the crossed product algebra 
\[R\rtimes\Gamma:=\{\sum_{\gamma\in\Gamma}r_\gamma\gamma\,:\,r_\gamma\in
R\},\]
with multiplication given by the distributive law and the relation
\[\gamma r=\gamma(r)\gamma,\quad\text{for $\gamma\in \Gamma$ and $r\in R$.}\]  
Let $X = \Irr \, R$.   Now $\Gamma$ acts on
$\Irr\, R$ and as above we get $\square$ for free.
Here we have
\[\widetilde
X_2^\square=\{(V,\tau)\,:\,V\in\Irr\,R,\;\tau\in\Irr^{\square(V)}(\Gamma_V)\}.\]

\begin{lem} \label{lem:Clifford_algebras} We have a canonical bijection
\[
\Irr(R \rtimes \Gamma) \simeq (\Irr \, R \q \Gamma)_2^{\square}.
\]
\end{lem}
\begin{proof}  The proof proceeds by comparing our construction with the 
theory of Clifford as stated in \cite[Theorem~A.6]{RamRam}. 
\end{proof}

\begin{notation} \label{not:rtimes} {\rm We shall denote by $\tau_1\rtimes\tau$ 
(resp. $V\rtimes\tau$) the element
of $\Irr(X\rtimes \Gamma)$ which corresponds to $(\tau_1,\tau)$ (resp.
$(V,\tau)$) by the bijection of Lemma~\ref{lem:Clifford}
(resp.~\ref{lem:Clifford_algebras}).}
\end{notation}
\section{\ Twisted extended quotients for Weyl groups}

Let $M$ be a reductive complex algebraic Lie group. Then $M$ may have a finite number of connected components,  $M^0$ is the identity component of $M$, and $\cW^{M^0}$ is the Weyl group of $M^0$:
\[
\cW^{M^0}: = N_{M^0}(T)/T
\]
where $T$ is a maximal torus of $M^0$.   We will need the analogue of the Weyl group for the possibly disconnected group $M$. 

\begin{lem} \label{lem:disconnected}
Let $M,\, M^0,\, T$ be as defined above. Then we have 
\[
\Nor_M(T)/T = \cW^{M^0} \rtimes \pi_0(M).
\]
\end{lem}
\begin{proof}
The group $\cW^{M^0}$ is a normal subgroup of $ \Nor_M(T)/T$. Indeed,
let $n\in\Nor_{M^0}(T)$ and let $n'\in\Nor_M(T)$, then $n'nn^{\prime-1}$
belongs to $M^0$ (since the latter is normal in $M$) and normalizes $T$, that is,
$n'nn^{\prime-1}\in\Nor_{M^0}(T)$. On the other hand,
$n'(nT)n^{\prime-1}=n'nn^{\prime-1}(n'Tn^{\prime-1})=n'nn^{\prime-1}T$.

Let $B$ be a Borel subgroup of $M^0$ containing $T$.   
Let $w\in \Nor_M(T)/T$. Then $wBw^{-1}$ is a Borel subgroup of $M^0$
(since, by definition, the Borel subgroups of an algebraic group are 
the maximal closed connected solvable subgroups). Moreover, $wBw^{-1}$  
contains $T$. 
In a connected reductive algebraic group, the intersection of two Borel 
subgroups always contains a maximal torus and the two Borel subgroups are 
conjugate by a element of the normalizer of that torus. Hence $B$ and
$wBw^{-1}$ are conjugate by an element $w_1$ of $\cW^{M^0}$.
It follows that $w_1^{-1}w$ normalises $B$. Hence
\[w_1^{-1}w\in \Nor_M(T)/T \cap \Nor_{M}(B)=\Nor_{M}(T,B)/T,\] 
that is, \[
\Nor_M(T)/T = \cW^{M^0}\cdot(\Nor_M(T,B)/T).\] 
Finally, we have
\[\cW^{M^0}\cap(\Nor_M(T,B)/T)=\Nor_{M^0}(T,B)/T=\{1\},\] 
since $\Nor_{M^0}(B)=B$ and $B\cap \Nor_{M^0}(T)=T$. This proves (1).

Now consider the following map:
\begin{align}\label{MM}
\Nor_{M}(T,B)/T\to M/M^0\quad\quad mT\mapsto mM^0.
\end{align}
It is injective. Indeed, let $m,m'\in\Nor_{M}(T,B)$ such that
$mM^0=m'M^0$. Then $m^{-1}m'\in M^0\cap\Nor_{M}(T,B)=\Nor_{M^0}(T,B)=T$
(as we have seen above). Hence $mT=m'T$.

On the other hand, let $m$ be an element in $M$. Then $m^{-1}Bm$ is a
Borel subgroup of $M^0$, hence there exists $m_1\in M^0$ such that
$m^{-1}Bm=m_1^{-1}Bm_1$. It follows that $m_1m^{-1}\in\Nor_M(B)$. Also
$m_1m^{-1}Tmm_1^{-1}$ is a torus of $M^0$ which is contained in 
$m_1m^{-1}Bmm_1^{-1}=B$. Hence $T$ and $m_1m^{-1}Tmm_1^{-1}$ are conjugate
in $B$: there is $b\in B$ such that $m_1m^{-1}Tmm_1^{-1}=b^{-1}Tb$. Then 
$n:=bm_1m^{-1}\in\Nor_M(T,B)$. It gives $m=n^{-1}bm_1$. Since $bm_1\in
M^0$, we obtain $mM^0=n^{-1}M^0$. Hence the map  (\ref{MM}) is surjective.
\end{proof}

\bigskip

Let $H$ denote a connected complex reductive group, let $T$ be a maximal torus in $H$.   Let 
\begin{align}\label{MMbis}
M: =  M(t) = \Cent_H(t)
\end{align}
denote the centralizer in $H$ of $t \in T$.  
\smallskip

The Weyl group of $H$ is denoted $\cW^H$. 

\begin{lem} \label{lem:centrals} 
Let $t \in T$. 
The isotropy subgroup $\cW^H_t$ is the group $\Nor_M(T)/T$, and we
have
\[
\cW^H_t = \cW^{M^0} \rtimes \pi_0(M)
\]
In the case when $H$ has simply-connected derived group, the
group $M$ is connected and $\cW^H_t$ is then the Weyl group of
$M^0$.
\end{lem}
\begin{proof} Let $t \in T$. Note that $H$ and $\Cent_H(t)$ have a common maximal torus $T$.  The we have
 \begin{align*}
 \cW^H_t &  = \{w \in \cW^H : w\cdot t = t\}\\
 & = \{w \in \cW^H : wtw^{-1} = t\}\\
 & = \{w \in \cW^H : wt = tw\}\\
 & = \cW^H \cap \Cent_H(t)\\
 & = \Nor_H(T)/T \cap \Cent_H(t)\\
& = \Nor_{\Cent_H(t)}(T)/T.
\end{align*}
The result follows by applying Lemma~\ref{lem:disconnected} with $M=\Cent_H(t)$.
 
If $H$ has simply-connected derived group, then the
 centralizer $\Cent_H(t)$ is connected by Steinberg's theorem
\cite[\S 8.8.7]{CG}. 
\end{proof}

The twisted extended quotient of the second kind now elucidates the extended quotient of the second kind, for we have 
\begin{align}
(T\q \cW^H)_2  & =  \{(t,\tau) : t \in T, \tau \in
\Irr(\cW_t^H)\}/\cW^H\label{EXT1}\\
\Irr\, \cW^H_t  & =  (\Irr \, \cW^{M^0}\q
\pi_0(M))_2^{\square}\label{EXT2}
\end{align}

\section{\ The principal series of $\cG$}
\label{sec:unram}
Let $\cG$ be the group of $F$-rational points of a connected reductive 
algebraic group defined over $F$. Let $\cT$ be a maximal torus in $\cG$ and let
let $\chi$ be a smooth irreducible character of $\cT$. 
In the case where $\chi$ is non-trivial, we will assume that $\cG$ has connected
center and the residual characteristic $p$ satisfies the hypothesis in
\cite[p.~379]{Roc}. 

Let $[\cT,\chi]_{\cG}$ be the 
inertial equivalence class of the pair $(\cT, \chi)$, see \cite{Renard}.  
We will write $\fs = [\cT,\chi]_{\cG}$ 
for this point in the Bernstein spectrum $\fB(\cG)$.   
Let 
\begin{equation} \label{eqn:Ws}
W^{\fs}  = \{w \in \cW : w\cdot \fs = \fs\}.
\end{equation}

Let $X$ denote the rational co-character group of $\cT$, identified with
the rational character group of $T$.  The group $\cG$ is split, so choose an isomorphism
\begin{align}\label{FF}
\cT \simeq F^{\times} \times \cdots \times F^{\times}.
\end{align}

$T$ identifies canonically with the unramified quasicharacters of $\cT$ so the above chosen isomorphism determines an isomorphism
\[
T \simeq \Cset^{\times} \times \cdots \times \Cset^{\times}.
\]
Now $\chi$ maps $\cT$ to $\Cset^{\times}$ so the isomorphism (\ref{FF}) factors $\chi$ as $\chi = \chi_1 \cdots \chi_l$ where each $\chi_j$ is a smooth character of $F^{\times}$.  Let  $U_F$ denote the group of units in $\mathfrak{o}_F$.   We define $\hat{\chi}$ as follows:
\[
\hat{\chi} : U_F \to T, \quad \quad  y \mapsto (\chi_1(y), \ldots, \chi_l(y)).
\]
 This construction is canonical, \ie $\hat{\chi}$ depends only on $\chi$ and not on the choice of isomorphism  (\ref{FF}). 

\medskip

The image of $\hat{\chi}$ is a finite abelian subgroup of $T$.   Let $H$ denote the centralizer in $G$ of the image of $\hat{\chi}$:
\begin{equation} \label{eqn:H}
H= \Cent_G(\im \hat{\chi}). 
\end{equation}
 
\begin{lem} \label{lem:Roche}
The group $H$ is connected. The 
group $W^\fs$ defined in~{\rm (\ref{eqn:Ws})} is a Weyl group, and it
is the Weyl group of $H$:
\[W^\fs= \cW^H.\]
\end{lem} 
\begin{proof}
We will treat separately the cases where $\chi$ is trivial and where is
non-trivial.

\noindent 
$\bullet$ 
We consider first the case where $\fs=[\cT,1]_\cG$ (that is,
the character $\chi$ is trivial). 
Then \[H=G\quad\text{and}\quad W^\fs=\cW.\]
The result follows.

\noindent
$\bullet$
We assume now that $\chi\ne 1$. The proof will follow the same lines as in
\cite{Roc}.

Let $R(G,T)$ denote the root system of $G$. 
The group $\Cent_G(\im\hat\chi)$ is the reductive subgroup of $G$ generated 
by $T$ and those root groups $U_\alpha$ for which $\alpha\in R(G,T)$ has 
trivial restriction to $\im\hat\chi$ together with 
those Weyl group representatives $n_w\in\Nor_G(T)$ ($w\in\cW$) for which $w(t)=t$ for all $t\in\im\hat{\chi}$.
The identity component of $\Cent_G(\im\hat\chi)$ is generated by 
$T$ and those
root groups $U_\alpha$ for which $\alpha$ has trivial restriction to
$\im\hat\chi$ (see \cite[\S~4.1]{SpringerSteinberg}).

Let $\fp_F$ and $k_F$ denote the maximal ideal of the ring of integers of $F$ and 
the residual field of $F$, respectively.
Let $\cE$ denote the maximal compact subgroup of $\cT$:
\[\cE\simeq U_F\times\cdots \times U_F,\] 
and let $\cE_\unip\simeq \fp_F\times\cdots\times\fp_F$ and $\cE_\red\simeq k_F^\times\times\cdots\times k_F^\times$ denote the pro-$p$-unipotent radical 
and the reductive quotient of $\cE$, respectively. The exact sequence
\[1\to \cE_\unip\to \cE\to \cE_\red\to 1\]
is split. Hence the restriction to $\cE$ of $\chi$ factors as
$\chi=\chi_\unip\cdot \chi_\red$ where $\chi_\unip$ 
coincides with the restriction 
of $\chi$ to $\cE_\unip$ and $\chi_\red$ is trivial on $\cE_\unip$. This implies a
similar decomposition of the character $\hat\chi$:
\[\hat\chi=\hat \chi_\unip\cdot\hat\chi_\red.\]
Since $\im\hat\chi=\im\hat \chi_\unip\,\cdot\,\im\hat\chi_\red$, we have
\[H=\Cent_{H_\unip}(\im\hat\chi_\red),\quad\text{where}\quad
H_\unip:=\Cent_G(\im\hat\chi_\unip).\]

Since $G$ has simply connected derived group, it follows from Steinberg's
connectedness theorem \cite{Steinberg} that the group $H_\unip$
is connected.
Moreover, since $\im \hat\chi_\unip$ consists of elements of $p$-power order, where $p$ is
the characteristic of $k_F$, and $p$ is not a torsion prime for $R(T,G)$,
the derived group of $H_\unip$ is simply connected (see \cite{Steinbergtorsion}). 

On the other hand $\im\hat\chi_\red$ is cyclic, since $k_F^\times$ is.
Applying Steinberg's connectedness theorem in the group $H_\unip$, we get 
that $H$ itself is connected. The first assertion of the Lemma is proved.

Now let $W_{\hat\chi}$ denote the stabilizer in $\cW$ of $\im\hat\chi$ and let
$\cW_{\hat\chi}$ be the normal subgroup generated by those reflections
$s_\alpha$ such that $\alpha$ has trivial restriction to
$\im\hat\chi$.
Then
\[W_{\hat\chi}/\cW_{\hat\chi}\,\simeq\,H/H^0=\{1\}.\]
Moreover, $\cW_{\hat\chi}$ coincides with the Weyl group of $H$.
On the other hand, 
for every $w\in\cW$ the condition $w(t)=t$ for all $t\in\im\hat\chi$ 
is equivalent to the condition $w\in W^\fs$. 
Hence we get \[W^\fs=W_{\hat\chi}=\cW_{\hat\chi}=\cW^H.\]
%Also the condition that $\alpha$ has trivial restriction to
%$\im\hat\chi$ is equivalent to the condition that $\chi\circ\alpha$ has
%trivial restriction to $U_F$.
%\[W^\fs=\cW^\fs\rtimes C^\fs.\] 
\end{proof}

\begin{rem}
{\rm Note that $H$ itself does not
have simply-connected derived group in general (for instance, if $G$ is
the exceptional group of type $\rG_2$, and $\chi$ is the tensor square
of a ramified quadratic character of $F^\times$ then $H=\SO(4,\Cset)$).}
\end{rem}

We summarize our findings as follows:
\begin{align}\label{fs}
W^{\fs} & = \cW^H\\
(T\q \cW^H)_2  & =  \{(t,\tau) : t \in T, \tau \in \Irr(\cW_t^H)\}/\cW^H\\
\Irr\, \cW^H_t  & =  (\Irr \, \cW^{M^0}\q \pi_0(M))_2^{\square}
\end{align}

\section{\ The cell decomposition} \label{sec:celldec}

To begin the construction of the extended quotient, we choose a 
semisimple element $t \in T$. 
 We recall the definitions in Eqn.~(\ref{MMbis}) and Eqn.~(\ref{eqn:H}):
\[
M = M(t) = \Cent_H(t), \quad \quad H = \Cent_G(\im \, \widehat{\chi}).
\]

 Let 
\begin{align}
A_x: = \pi_0 (\Cent_{M^0}(x)).
\end{align}
Let $\bB_x$ denote the variety of Borel
subalgebras of $LM^0$ that contain $x$.   All the irreducible components of $\bB_x$ have the same dimension $d(x)$
over $\Rset$, see \cite[Corollary 3.3.24]{CG}.   
The finite group $A_x$  acts on the  set
of irreducible components of $\bB_x$ \cite[p. 161]{CG}.
%The Springer fibre $\bB_x$ is connected \cite[6.7.16]{CG}.

The Springer correspondence yields a one-to-one correspondence
\begin{equation} \label{eqn:Springercor}
(x,\varrho)\mapsto \tau(x,\varrho)\end{equation}
between the set of $M^0$-conjugacy classes of pairs $(x,\varrho)$ formed by a
nilpotent element $x \in LM^0$ and an irreducible representation 
$\varrho$ of $A=A_x$ which occurs in $\H_{d(x)}(\bB_x, \Cset)$  
and the set of isomorphism classes of irreducible representations of the Weyl 
group $\cW^{M^0}$.

\begin{rem} \label{rem:Springer}
{\rm
The Springer correspondence that we are considering in this article
 coincides with that constructed by Springer for a reductive group
over a field of positive characteristic and is obtained
from the correspondence constructed by Lusztig by tensoring the latter by
the sign representation  of $\cW^{M^0}$ (see \cite{Hot}).
}\end{rem}

\smallskip

Let $X(T)$ denote the group of characters of $T$.
Recall that $R(G,T)$ is the root system of $G$. 
We have seen in Lemma~\ref{lem:Roche} that $W^\fs$
is the Weyl group of the connected group $H$ defined in
Eqn.~(\ref{eqn:H}). The group $H$ has root datum
\[(X(T),R^{\fs},X_\bullet(T),R^{\fs\vee})\]
where $R^\fs$ is the root system:
\[R^\fs=\left\{\alpha\in
R(G,T)\;:\;(\chi\circ\alpha)_{|U_F}=1\right\},\]
and $W^{\fs}$ coincides with the Weyl group of $R^\fs$.
Hence $X(T)\rtimes W^\fs$ is the extended affine Weyl group of $H$.

%\[\cW^\fs:=\left\{s_\alpha\;:\;\alpha\in R^\fs\right\}.\]
%Let $R(G,T)_+$ be the set of positive roots in $R(G,T)$. Setting \[R^{\fs}_+:=
%R^{\fs}\cap R(G,T)_+\quad\text{and}\quad C^\fs:=\left\{w\in
%W^\fs\;:\;w(R_+^\fs)=R_+^\fs\right\},\] we have (see \cite[Lemma~8.1]{Roc}):
Let $\bc$ be a two-sided cell of $X(T)\rtimes W^\fs$ and let $\cU$ be the 
unipotent class in $H$ which corresponds to $\bc$ by the Lusztig bijection 
\cite[Theorem~4.8]{LCellsIV}.   

%Recall that, for $t\in T$, we have set $M=\Cent_H(t)$, and that we have 
%denoted by $\tau(x,\rho)\in\Irr(\cW^{M^0})$ the image by the Springer 
%correspondence~(\ref{eqn:Springercor}) of
%the $M^0$-conjugacy class of the pair $(x,\varrho)$ formed
%by a nilpotent element $x \in LM^0$ and an irreducible representation
%$\varrho$ of $A=A_x$ which occurs in $\H_{d(x)}(\bB_x, \Cset)$.

We have seen in Lemma~\ref{lem:centrals} that
\[W^\fs_t= \cW^{M^0} \rtimes \pi_0(M),\]
and, from Eqn.~(\ref{EXT2}), we know that every $\tau\in  \Irr(W_t^\fs)$ can 
be written as
\[\tau=\tau(x,\varrho)\rtimes\psi\quad\text{where
$\psi\in\Irr^{\square(\tau(x,\varrho))}(\pi_0(M)_{\tau(x,\varrho)})$.}\]

\smallskip
Recall the definition:
\[
\widetilde{T}_2: = \{(t,\tau) \,:\, t \in T, \tau \in \Irr(W_t^\fs)\},
\]
\[
(T\q W^\fs)_2: = \widetilde{T}_2/W^\fs.
\]
Then we set 
\begin{align} \label{eqn:T2c}
\widetilde{T}_2^{\bc}& := \{(t,\tau(x,\varrho)\rtimes\psi)\in\widetilde{T}_2 
\,:\, \exp x \in \cU\},\\
(T\q W^\fs)_2^{\bc} & := \widetilde{T}_2^{\bc}/W^\fs.\label{eqn:T2cbis}
\end{align}
 This determines the cell-decomposition 
 \begin{align}\label{cell}
 (T\q W^{\fs})_2 = \bigsqcup (T\q W^{\fs})_2^{\bc}
 \end{align}
where $\bc$ runs over the two-sided cells of $X(T)\rtimes W^\fs$.

\begin{rem} \label{rem:celldec}
{\rm Let $\nu\colon T\q W^\fs \to (T\q W^\fs)_2$ a chosen bijection 
obtained from a $c$-$\Irr$ system in the situation of Lemma~\ref{lem:csyst}. 
We define:
\[(T\q W^{\fs})^{\bc}:=\nu^{-1}((T\q W^{\fs})_2^{\bc}).\]
It provides a non-canonical cell-decomposition of the extended quotient:
\begin{align} \label{eqn:cellddec}
 T\q W^{\fs} = \bigsqcup (T\q W^{\fs})^{\bc}.
 \end{align}
}\end{rem}   

\section{\ The Langlands parameter $\Phi$}
Let $\mathbf{W}_F$ denote the Weil group of $F$, let $\mathbf{I}_F$ be the inertia
subgroup of $\mathbf{W}_F$. 
% Let $\PPhi_F \subset W_F$ denote a geometric Frobenius (a generator of $W_F/I_F \simeq \Zset$).
% We have $W_F/I_F = <\Frob>$. We will think of this as a multiplicative group, with identity element $1$.
Let $\mathbf{W}_F^{\der}$ denote the closure of the commutator subgroup of $\mathbf{W}_F$, and write $\mathbf{W}_F^{\ab} = \mathbf{W}_F/\mathbf{W}^{\der}_F$. 
The group of units in $\mathfrak{o}_F$ will be denoted $U_F$.

We recall the Artin reciprocity map $\mathbf{a}_F : \mathbf{W}_F \to F^{\times}$ which has the following properties (local class field theory):
\begin{enumerate}
\item The map $\mathbf{a}_F$ induces a topological isomorphism $\mathbf{W}^{\ab}_F \simeq F^{\times}$.
\item An element $x \in \mathbf{W}_F$ is a geometric Frobenius if and only if $\mathbf{a}_F(x)$ is a prime element $\varpi_F$ of $F$.
\item We have $\mathbf{a}_F(\mathbf{I}_F) = U_F$.
\end{enumerate}

We now consider the  principal series of $\cG$.     We recall that  
 $\mathcal{G}$ denotes a connected reductive split $p$-adic group with 
maximal split torus $\mathcal{T}$, and
$G$, $T$ denote the Langlands dual of $\mathcal{G}$, $\mathcal{T}$.
We recall that, in the case of a non-trivial inducing character, 
we assume in addition that $\cG$ has connected
center and the residual characteristic of $F$ satisfies the hypothesis in
\cite[p.~379]{Roc}.

Next, we consider  conjugacy classes in
$G$ of pairs $(\Phi,\rho)$ such that $\Phi$ is a continuous morphism
\[
\Phi\colon \mathbf{W}_F\times \SL(2,\Cset) \to G\] which is
rational on $\SL(2,\Cset)$ and such that $\Phi(\mathbf{W}_F)$ consists of semisimple 
element in
$G$, and $\rho$ will be defined in the next section.

Choose a Borel subgroup $B_2$ in $\SL(2,\Cset)$.
Let $\bB^{\Phi}$ denote the variety of Borel subgroups of $G$ containing
$\Phi(\mathbf{W}_F \times B_2)$.
The variety $\bB^{\Phi}$ is non-empty if and only if $\Phi$ factors
through $W_F^{\ab}$, (see \cite[\S~4.2]{R}).  In that case,  we view  the domain of $\Phi$ to be $F^{\times}$: 
\[
\Phi\colon F^{\times} \times \SL(2,\Cset) \to G.\]

\smallskip
We will now build such a continuous morphism $\Phi$ from data coming from
the extended quotient of second kind. 

We now work with the Jacobson-Morozov theorem \cite[p. 183]{CG}.  Let
$e_0$ be the standard nilpotent matrix in $\mathfrak{sl}(2,\Cset)$:
\[e_0 = \left(
\begin{array}{cc}
0 & 1 \\
0 & 0 \end{array}\right). \]
Let $x$ be a nilpotent element in $LM^0$. 
There exists a rational homomorphism 
\begin{equation} \label{eqn:gamt}
\gamma \colon \SL(2, \Cset) \to M^0\end{equation}
such that its differential $\mathfrak{sl}(2,\Cset) \to LM^0$
sends $e_0$ to $x$, see \cite[\S 3.7.4]{CG}.

The rational homomorphism $\gamma$
depends only on the unipotent class in $M^0$ containing $\exp x$. 

Let $\alpha \in \Cset^{\times}$.  Define the following matrix in $\SL(2,\Cset)$:
\[
Y_{\alpha} = 
\left(\begin{array}{cc}
\alpha & 0\\
0 & \alpha^{-1}\end{array}\right)
.\]
Then $\gamma$ determines a cocharacter  
\[h\colon \Cset^\times\to M^0\] by setting 
\begin{equation}\label{hh}
h(\alpha):=\gamma(Y_{\alpha})\quad\text{for $\alpha\in\Cset^\times$}.
\end{equation}
The cocharacter $h$ is \emph{associated} to $x$, see \cite[Rem.~5.5]{J}
or \cite[Rem.2.12]{FR}
(see \cite[Def.~2.9]{FR} or \cite[\S~6]{McN} for the definition of
associated). In other words, $\gamma$ is an \emph{optimal}
$\SL_2$-homomorphism for $x$ (see \cite[Def.~32]{McN}).

On the other hand, the cocharacter $h$ depends only on the nilpotent class in $M^0$ containing 
$x$.

\begin{lem} \label{lem:cocharindep}
The cocharacter $h$ can be 
identified with a cocharacter of $H$, and hence depends 
only on the nilpotent class in $H$ containing $x$.
\end{lem}
\begin{proof}
Recall J.C.~Jantzen's result \cite[Claim~5.12]{J} 
(see also \cite{FR} for a related study in
positive good characteristic): For any connected reductive 
subgroup $H_2$ of an arbitrary connected complex Lie group $H_1$, the 
cocharacters of $H_2$ associated to a nilpotent element $x\in LH_2$ are 
precisely the cocharacters of $H_1$ associated to $x$ which take values in 
$H_2$.  

Applying this with $H_1=H$ and $H_2=M^0$, we get that $h$ can be
identified with a cocharacter of $H$, and hence depends 
only on the nilpotent class in $H$ containing $x$.
\end{proof}

From now on we view $h$ as a cocharacter of $H$ associated to $x$.
We can consider $\gamma\colon \SL(2,\Cset)\to M^0\subset H$ as an
optimal $\SL_2$-homomorphism for $x$.

Now, let $\gamma'\colon
\SL(2,\Cset)\to H$ be an optimal $\SL_2$-homomorphisms for $x$ such that
\[\gamma'(Y_\alpha)=\gamma(Y_\alpha)\quad\text{for
$\alpha\in\Cset^\times$.}\]
Then $\gamma'=\gamma$ (see \cite[Prop.~42]{McN}).
Hence $\gamma$ depends only on the unipotent class in $H$ containing $\exp
x$.

\smallskip

Define the $L$-parameter  $\Phi$ as follows:
\begin{eqnarray} \label{Phi}
\quad
&\Phi \colon F^{\times} \times \SL(2,\Cset) \to G, \quad\quad  (u\varpi_F^n,Y) 
\mapsto \hchi(u)\cdot t^n\cdot \gamma(Y) \label{eqn:Phi}
\end{eqnarray}
and the real tempered parameter $\Xi$ as follows:
\begin{eqnarray} 
\quad
&\Xi \colon  F^{\times} \times \SL(2,\Cset) \to M^0, \quad \quad (u\varpi_F^n,Y) 
\mapsto \gamma(Y)
\label{eqn:Xi}
\end{eqnarray}
for all  $u \in U_F, \; n \in \Zset,\; Y \in \SL(2,\Cset)$.
\medskip

Note that the definition of $\Phi$ uses the appropriate data: the semisimple element $t \in T$, the character $\widehat{\chi}$ (which depends on the point $\fs$), and the 
homomorphism $\gamma$ (which depends on the Springer parameter $x$).    
Eqn.~(\ref{Phi}) determines the first of the Reeder parameters $(\Phi,\rho)$.  
 
 We turn next to the construction of $\rho$. 

\section{\ The  parameter $\rho$}

The centralizer in $G$ of the image of $\Phi$ 
acts naturally on the variety $\bB^{\Phi}$, and hence on the 
singular homology of $\H_*(\bB^{\Phi},\Cset)$.  This action will factor through
\[
A_{\Phi}: = \pi_0(\Cent_G(\im \Phi))
\]
 and then $\rho$ is an irreducible representation of $A_{\Phi}$ which appears in  $\H_*(\bB^{\Phi},\Cset)$. 
We have (see \cite[Lemma~4.3.1]{R}):
\[A_{\Phi}\simeq  A^+=A_x^+: =\pi_0(\Cent_M(x)).\]
% We have the short exact sequences
% \begin{align}
% 1 \to M^0 \to M \to \pi_0(M) \to 0
% \end{align}
% and
% \begin{align}
% 1 \to \Cent_{M^0}(\im \Xi) \to \Cent_M(\im \Xi) \to \pi_0(M) \to 0.
% \end{align}
We clearly have
 \[
 \Cent_G(\im \Phi) = \Cent_{M}(\im \Xi)
 \]
 Let
 \[
 A_{\Xi}: = \pi_0(C_{M^0}(\im\, \Xi).
 \]
 
 \begin{lem} \label{lem:AAA}
We have
\[
A_x = A_{\Xi}.
\]
\end{lem}

\begin{proof}   According to \cite[\S  3.7.23]{CG},  we have
\[
\Cent_{M^0}(x)  = \Cent_{M^0}(\im \, \Xi)\cdot U
\] 
with $U$ the unipotent radical of $\Cent_{M^0}(x)$. Now $U$ is contractible
via the map
\[
[0,1] \times U \to U, \quad \quad (\lambda, \exp Y) \mapsto \exp( \lambda
Y)
\]
for all $Y \in \fn$ with $\exp \fn = U$.
\end{proof}

Lemma~\ref{lem:AAA} allows us to define
\[
A:  = A_x = A_{\Xi}.
\]

\medskip

Let $\cM(t)$ denote a {\it predual} of $M^0$, \ie
$M^0$ is the Langlands dual of $\cM(t)$. 
Let $\bB^{\Xi}$ denote the variety of the Borel
subgroups of $M^0$ which contain $\Xi(W_F \times B_2) = \gamma(B_2)$.

\begin{defn} If a group $A$ acts on the variety $\mathbf{X}$, let
$\cR(A,\mathbf{X})$ denote the  set of equivalence classes of irreducible representations of $A$
appearing in the homology $\H_*(\mathbf{X})$, as in \cite[p.118]{R}.  Let
$\cR_{\ttop}(A, \mathbf{X})$ denote the set of equivalence classes of irreducible
representations of $A$ appearing in the top homology of $\mathbf{X}$.
\end{defn}

\begin{lem} \label{lem:bije}
We have
\[
\cR_{\ttop}(A, \bB_x) = \cR(A, \bB^{\Xi}).
\]
\end{lem}
\begin{proof}  Let, as before,  $\tau$ be an irreducible representation
of $\cW^{M^0}$.   Let $(x,\varrho)$ be the Springer parameter attached to
$\tau$.   Define $\Xi$ as in Eqn.~(\ref{eqn:Xi}).  Note that $\Xi$ depends on the
morphism $\gamma$, which in turn depends on the nilpotent element $x \in
LM^0$.

Then $\Xi$ is a real tempered $L$-parameter for the $p$-adic group
$\cM(t)$, see \cite[3.18]{BM}.  According to 
\cite[\S 10.13]{Lu}, \cite{BM}, there is a bijection between
 Springer parameters and Reeder parameters:
\begin{equation}  \label{eqnarray:bij}
(x,\varrho) \mapsto (\Xi, \varrho).
\end{equation}
Now $\varrho$ is an irreducible representation of $A$ which appears
simultaneously in $\H_{d(x)}(\bB_x, \Cset)$ and $\H_*(\bB^{\Xi}, \Cset)$.
%ICI!
If $\H_{d(x)}(B_x,\Cset)_\varrho=0$ then $\H_i(B_x,\Cset)_\varrho=0$ for any $i\ge 0$, where
$\H_i(B_x,\Cset)_\varrho$ denotes the $\rho$-isotypic subspace of $\H_i(B_x,\Cset)$ (see
\cite[bottom of page~296 and Remark 6.5]{Shoji}).

\end{proof}

Define 
\begin{eqnarray} 
\quad
&\Upsilon \colon F^{\times}\times \SL(2,\Cset) \to H, \quad\quad  (w\varpi_F^n,Y) 
\mapsto \hchi(w)\cdot t^n \cdot \gamma(Y)
\label{eqn:Upsilon}
\end{eqnarray}
\begin{eqnarray} 
\quad  
&\;\; \Psi \colon F^{\times} \times \SL(2,\Cset) \to M^0, \quad\quad 
(u\varpi_F^n,Y) \mapsto \hchi(w)\cdot t^n \cdot \gamma(Y)
\label{eqn:Psi}
\end{eqnarray}

\begin{lem} \label{lem:PX} We have
\[
\bB^{\Psi} = \bB^{\Xi}.
\]
\end{lem}
\begin{proof}  
We note that
\[
S_{\Psi} = < t > \gamma(B_2), \quad \quad S_{\Xi} = \gamma(B_2)
\]

Let $\fb$ denote a Borel subgroup of the reductive group $M$.  Since $\fb$ is maximal among the connected solvable subgroups of $M$,  we have 
$\fb \subset M^0$.   Then we have $\fb = T_{\fb}U_{\fb}$ with $T_{\fb}$ a maximal torus in $M^0$, and 
$U_{\fb}$ the unipotent radical of $\fb$.  Note that $T_{\fb} \subset M^0$.  Therefore $yt = ty$ for all $y \in T_{\fb}$. This means that $t$ centralizes $T_{\fb}$, i.e. $t \in Z(T_{\fb})$. In a connected Lie group such as $M^0$, we have 
\[
Z(T_{\fb}) = T_{\fb}\]
 so that $t \in T_{\fb}$. Since $T_{\fb}$ is a group, it follows that $< t > \, \subset T_{\fb}$. 

As a consequence, we have
\[
\fb \supset \,  < t > \gamma(B_2) \iff \fb  \supset \gamma(B_2).
\]
\end{proof}

\medskip

Let $\bB^{\Upsilon}$ denote the variety of Borel subgroups of $H$ containing the image
$ \Upsilon(W_F \times B_2)$.

\begin{lem} \label{lem:UP} We have
\[
\mathcal{R}(A, \bB^{\Upsilon}) = \mathcal{R}(A, \bB^{\Psi})
\]
\end{lem}

\begin{proof}    
We denote the Lie algebra of $M^0$ by $LM^0$.   We note that the codomain of $\Psi$ is $M^0$.  

Let $\bB^t$ denote the variety of all Borel subgroups of $H$ which contain $t$.  Let $B \in \bB^t$. 
Then $B \cap M^0$ is a Borel subgroup of $M^0$.    

The proof  in \cite[p.471]{CG} depends on the fact that $M^0$ is connected, and also on
a triangular decomposition of $\Lie(M^0)$:
\[
\Lie\,M^0 = \fn^t \oplus \ft \oplus \fn_{-}^t
\]
from which it follows that $\Lie\, B \cap \Lie\, M^0 = \fn^t \oplus \ft$ is a Borel subalgebra in $\Lie \, M^0$, where $\fn^t$ denotes 
the centralizer of $t$. 

There is a canonical map 
\begin{align} \label{eqn:(7)} 
\bB^t \to \Flag \, M^0, \quad B \mapsto B \cap M^0
\end{align}
Now $M^0$ acts by conjugation on $\bB^t$. We have
\begin{align}
\bB^t = \bB_1 \sqcup \bB_2 \sqcup \cdots \sqcup \bB_m
\end{align}
a disjoint union of $M^0$-orbits,  see \cite[Prop. 8.8.7]{CG}. These orbits are the connected components of $\bB^t$, and the irreducible components of the projective variety
$\bB^t$. The above map~(\ref{eqn:(7)}), restricted to any one of these orbits, is a bijection from the $M^0$-orbit onto $\Flag \, M^0$ and is $M^0$-equivariant. It is then clear that 
\[
\bB_j^{\Upsilon} \simeq \Flag \, (M^0)^{\Psi}
\]
for each $1 \leq j \leq m$.  We also have $t \in S_\Upsilon = S_{\Psi}$.  Now
\[
\bB^{\Upsilon} =  (\bB^t)^{\Upsilon} = (\bB^t)^{\Psi}
\]
and then
\[
\H_*(\bB^{\Upsilon}, \Cset) = \H_*(\bB_1^{\Psi}, \Cset) \oplus \cdots \oplus \H_*(\bB_m^{\Psi}, \Cset)
\]
a direct sum of \emph{equivalent} $A$-modules.
%it follows that
%\[\bB^{\phi}=(\bB^t)^{\phi}=(\bB^t)^{\Psi}=\bB_1^{\Psi}\cup\cdots \cup\bB_m^{\Psi}.\]
Hence $\varrho$
occurs in $\H_*( \bB^{\Upsilon},\Cset)$ if and only if it occurs
$\H_*(\bB^{\Psi}, \Cset)$. 
\end{proof} 

\medskip
Recall that 
\[A^+=A_x^+=\pi_0(\Cent_M(x))\quad\text{and}\quad A=A_x=\pi_0(\Cent_{M^0}(x)).\]

\begin{lem} \label{lem:inertias}
Let $\pi_0(M)_x$ denote the isotropy group of $x$ in $\pi_0(M)$.
\begin{itemize}
\item[{\rm (1)}]
We have $A \rtimes \pi_0(M)_x = A^+$, and hence:
\[
 (\Irr(A) \q \pi_0(M)_x)_2^{\square} \simeq
\Irr(A^+), \quad \quad (\varrho, \psi) \mapsto \varrho
\rtimes \psi=:\rho
 .\]
\item[{\rm (2)}]
Moreover, we have
\[
\varrho \in \cR_{\ttop}(A, \bB_x) \iff \rho \in \cR(A^+,\bB^{\Phi})
.\]
\end{itemize}
\end{lem}

\begin{proof}
Let $\tH$ be a connected complex reductive group with 
simply connected derived subgroup and 
such that $H=\tH/Z$ where $Z$ is a finite
subgroup of the centre of $\tH$ (see \cite[\S~3]{R}). 
The natural projection
$p\colon\tH\to H$ is an isogeny. Let $\tilde t$ be a lift of $t$ in $\tH$, 
that is, $p(\tilde t)=\tilde t Z=t$. 
We recall the definition in Eqn.~(\ref{MMbis}):
\[M: =  M(t) = \Cent_H(t).\]
Let $LM^0$ denote the Lie algebra of $M^0$.  For $x\in LM^0$, let $\Cent_{M^0}(x)$
denote the centralizer of $x$ in $M^0$, via the adjoint representation of 
$M^0$ on $LM^0$. 
The isogeny $p$ induces an isomorphism from the set of nilpotent
elements in $L\tM$ to the set of nilpotent elements in $LM^0$,
and we will not distinguish between these
sets of nilpotent elements.
Let $x$ be a nilpotent element in $LM^0$.  

We set
\[\tM:=\Cent_\tH(\ttt)\quad\text{and}\quad\tM^+:= \Cent_{\tH}(t).\]
Then we have
\[\tM^+=\left\{\th\in\tH\,:\,\th (\ttt Z) \th^{-1}=\ttt Z\right\}=
\left\{\th\in\tH\,:\,(\th\,\ttt\,\th^{-1})Z\subset \ttt Z\right\}.\]
Since $\tH_\der$ is simply connected, the group $\tM$ is connected and \cite[\S~3.1]{R}
it is the identity component of $\tM^+$: 
\begin{equation} \label{eqn:iotacent}
\tM=(\tM^+)^0.\end{equation} 
Viewing $x$ as a nilpotent element
in $L\tM$, we can consider the centralizer of $x$ in $\tM^+$. We have
(see \cite[\S~3.3]{R}):
\[\pi_0(\Cent_{\tM^+}(x))\simeq A_x^+.\]
(In particular, taking $x=0$, we get $\pi_0(\tM^+)\simeq \pi_0(M)$.) 

Then \cite[\S~3.2-3.3]{R}, we have an exact sequence
 \begin{align}
 1 \to A_x \to A_x^+ \to \pi_0(M)_x \to 0.
 \end{align}
 so we have 
 \[ A_x \rtimes \pi_0(M)_x = A_x^+. \]
This gives (1).
  
Let $u:=\exp(x)\in M^0$.
Let $\tu$ denote a lift of $u$ in $\tH$. Since $ut=tu$, by applying $p$,
we get 
$\tu\tilde t=\tilde t\tu$, that is, $\tu\in \tM$.
It follows from \cite[Lemma~3.5.3]{R} that 
\[(\tilde t,\tu,\varrho,\psi) \mapsto (t, u,
\rho)\] induces a bijection between
$H$-conjugacy classes of quadruples $(\tilde t,\tu,\varrho,\psi)$ where 
\[\varrho\in\cR(A,\bB^\Upsilon)\quad\text{and}\quad
\psi\in\Irr^{\square(\varrho)}((A^+/A)_\varrho),\]
and $H$-conjugacy classes of triples $(t,u,\rho)$, where 
$\rho\in\cR(A^+,\bB^\Upsilon)$ is such that the restriction of $\rho$
to $A$ contains $\varrho$. 

By combining Lemmas~\ref{lem:bije}, \ref{lem:PX}, \ref{lem:UP}, we get
\[\cR(A,\bB^\Upsilon)=\cR_{\ttop}(A,\bB_x).\]
Also by \cite[Lemma~4.4.1]{R}
we have \[\cR(A^+,\bB^\Upsilon)=\cR(A^+,\bB^\Phi).\]
The result follows.
\end{proof}

\begin{lem} \label{lem:inertias1}
For every $\varrho\in\cR_{\ttop}(A,\bB_x)$, we have
\[\pi_0(M)_{\tau(x,\varrho)}=(\pi_0(M)_x)_{\varrho}=(A^+/A)_\varrho.\]
\end{lem} 
\begin{proof} 
We have
\begin{equation} \label{eqn: 22}
{}^a\tau(x,\varrho)=\tau(axa^{-1},{}^a\varrho)\quad\text{for any
$a\in\pi_0(M)$.}\end{equation}
Hence 
\[a\in \pi_0(M)_{\tau(x,\varrho)}\quad\Longleftrightarrow\quad
a\in \pi_0(M)_x\cap \pi_0(M)_\varrho.\]
In other words, we have
\begin{equation} \label{eqn: AA}
\pi_0(M)_{\tau(x,\varrho)}=(\pi_0(M)_x)_{\varrho}=(A^+/A)_\varrho,\end{equation}
by using Lemma~\ref{lem:inertias}.
\end{proof}
%\begin{lem} \label{lem:squares}
%For every $\varrho\in\cR_{\ttop}(A,\bB_x)$, we have
%\begin{equation} \label{eqn:equalsquares}
%\square(\tau(x,\varrho))=\square(\varrho),\end{equation}
%where on the L.H.S. (resp. R.H.S.), $\square$ means the function defined 
%at the end of Section~\ref{sec: teq} which assigns an element in 
%$\H^2((A^+/A)_\varrho,\Cset^\times)$ to $\tau(x,\varrho)$ (resp. $\varrho$).
%\end{lem}
%In order to prove Lemma~\ref{lem:squares}, we will first check that a 
%similar statement holds in the framework of Hecke algebras.

Recall the $p$-adic group $\cM(t)$, a predual of $M^0$.
Let $\cH(\cM(t))$ be its Iwahori-Hecke algebra.

The group of components $\pi_0(M)$ acts on both $\cM(t)$ and
$\Cset[\cW^{M^0}]$, and we can form the corresponding crossed product
algebras. 
Then, by applying Lemma~\ref{lem:Clifford_algebras}, we obtain:
\begin{align*} 
\Irr(\cH(\cM(t)) \rtimes \pi_0(M)) \simeq (\Irr \, \cH(\cM(t)) \q 
\pi_0(M))_2^{\square}\cr 
\Irr(\Cset[\cW^{M^0}] \rtimes \pi_0(M)) \simeq (\Irr \,\Cset[\cW^{M^0} 
\q \pi_0(M))_2^{\square}.
\end{align*}
Let
\begin{equation} \label{eqn:Hrc} 
\Irr(\cH(\cM(t)))^{\tempere}_{\reel}\,\subset
\Irr(\cH(\cM(t))\end{equation}
and
\begin{equation} \label{eqn:Hrcbis}
\Irr(\cH(\cM(t)) \rtimes \pi_0(M))^{\tempere}_{\reel}\,\subset
\Irr(\cH(\cM(t)) \rtimes \pi_0(M)) 
\end{equation}
denote the set of (equivalence classes) of real tempered irreducible 
representations of $\cH(\cM(t))$ and 
of $\cH(\cM(t)) \rtimes \pi_0(M)$, respectively. 

For $\varrho\in\cR_{\ttop}(A,\bB_x)$, let
$V(x,\varrho)\in \Irr(\cH(\cM(t)))^{\tempere}_{\reel}$ 
be the representation of $\cH(\cM(t))$ with Reeder parameter equal to 
the image of $(x,\varrho)$ by the bijection~{\rm (\ref{eqnarray:bij})}. 
Then the map $\tau(x,\varrho)\mapsto V(x,\varrho)$ provides a bijection
\[\Irr(\Cset[\cW^{M^0}]) \to \Irr(\cH(\cM(t)))^{\tempere}_{\reel}.\]
\begin{lem} \label{lem:HW}
Let $\varrho\in\cR_{\ttop}(A,\bB_x)$.
We have:
\begin{equation} \label{eqn:equalinertias1}
(\pi_0(M))_{V(x,\varrho)}=(\pi_0(M))_{\tau(x,\varrho)}
=(A^+/A)_\varrho,\end{equation}
and 
\begin{equation} \label{eqn:equalsquares1}
\square(V(x,\varrho))=\square(\tau(x,\varrho)),\end{equation}
where on the L.H.S. (resp. R.H.S.), $\square$ means the function  
which assigns an element in 
$\H^2((A^+/A)_\varrho,\Cset^\times)$ to $V(x,\varrho)$ (resp. 
$\tau(x,\varrho)$).

As a consequence, the map 
\begin{equation} \label{intermbij}
\tau(x,\varrho)\rtimes\psi\mapsto V(x,\varrho)\rtimes\psi
\end{equation}
is a bijection
\[\Irr(\Cset[\cW^{M^0}] \rtimes \pi_0(M)) \,\to\,\Irr(\cH(\cM(t)) \rtimes
\pi_0(M))^{\tempere}_{\reel}.\]
\end{lem}
\begin{proof}
Let $\orb_{x}$ denote the nilpotent adjoint orbit which contains $x$.
Then the usual closure order on nilpotent adjoint orbits, which is 
defined as
\[ \orb_1\le\orb_2\quad\text{when $\orb_1\subset\overline{\orb_2}$},\]
induces the following partial order on the irreducible representations 
of $\cW^{M^0}$: 
\begin{equation} \label{eqn:ordering}
\tau(x_1,\varrho_1)\le\tau(x_2,\varrho_2)\quad\text{when
$\orb_{x_1}\le{\orb}_{x_2}$}.\end{equation}
In this partial order, due to our choice of the Springer correspondence
(see Remark~\ref{rem:Springer}), the trivial representation of $\cW^{M^0}$ 
is a minimal element and the sign representation is a maximal element.

The $\cW^{M^0}$-structure of $V(x,\varrho)$ is
\begin{equation} \label{eqn:Wstruct}
V(x,\varrho)|_{\cW^{M^0}}\,=\,\tau(x,\varrho)\,\oplus\,
\bigoplus_{(x_1,\varrho_1)\atop\tau(x,\varrho)<\tau(x,\varrho_1)}
m_{(x_1,\varrho_1)}\,\tau(x_1,\varrho_1),\end{equation}
where the $m_{(x_1,\varrho_1)}$ are non-negative integers. 
(In case where $M$ has connected centre, (\ref{eqn:Wstruct}) is implied by
\cite[Theorem~6.3~(1)]{BM}, the proof in the general case follows the
same lines.)
In particular, it follows from (\ref{eqn:Wstruct}) that
\begin{equation} \label{eqn:dim}
\dim_{\Cset}\Hom_{\cW^{M^0}}\left(\tau(x,\varrho),V(x,\varrho)\right)=1.
\end{equation}
Let $a\in \pi_0(M)$.
Eqn.~(\ref{eqn: 22}) implies that
\[{}^a V(x,\varrho)|_{\cW^{M^0}}\,=\,\tau(a \cdot x,{}^a\varrho)\,\oplus\,
\bigoplus_{(x_1,\varrho_1)\atop\tau(x,\varrho)\le\tau(x_1,\varrho_1)}
m_{(x_1,\varrho_1)}\,\tau(a\cdot x,{}^a\varrho_1).\]
Since $\tau(x,\varrho)\le\tau(x_1,\varrho_1)$ if and only if
$\tau(a\cdot x,{}^a\varrho)\le\tau(a\cdot  x_1,{}^a\varrho_1)$, it follows
that ${}^a V(x,\varrho)$ corresponds to the $M^0$-conjugacy class of
$(a\cdot x,{}^a\varrho)$ via the bijection induced by~(\ref{eqnarray:bij}).

Hence \[{}^a V(x,\varrho)\simeq V(x,\varrho)\quad\text{ if and only if }\quad
{}^a\tau(x,\varrho)\simeq\tau(x,\varrho).\] Hence Eqn.~(\ref{eqn:equalinertias1})
follows.

Now let $\left\{T_{a^{-1}}^V\,:\,a \in (A^+/A)_{\varrho}\right\}$ (resp.
$\left\{T_{a^-1}^\tau\,:\,a \in (A^+/A)_{\varrho}\right\}$) a family of
isomorphisms for $V=V(x,\varrho)$ (resp. $\tau=\tau(x,\varrho)$) which determines the
cocycle $c_V$ (resp. $c_\tau$).
We have
\[\Hom_{\cW^{M^0}}(\tau,V)\overset{{T_{a^{-1}}^V}}\longrightarrow
\Hom_{\cW^{M^0}}(\tau,{}^{a^{-1}}V)\overset{{T_{a^{-1}}^\tau}}\longrightarrow
\Hom_{\cW^{M^0}}({}^{{a}^{-1}}\tau,{}^{a^{-1}}V).\]
The composed map is given by a scalar, since
by Eqn.~(\ref{eqn:dim}) these spaces are one-dimensional. We normalize
the operator $T_{a^{-1}}^V$ so that this scalar equals to one. This forces $c_V$ and
$c_\tau$ to be equal. This implies 
\[\square(V(x,\varrho))=\square(
\tau(x,\varrho)).\]
\end{proof}

Define
$\Irr(\cH(H))_t$ to be the set of (equivalence classes) of 
irreducible representations of $\cH(H)$ that are parametrized by 
pairs $(\Phi,\rho)$ where $\Phi(\varpi_F,1)=t$. 

%ICI
\begin{lem} \label{lem:end}
There a bijection
\[\Irr(\cM(t)\rtimes\pi_0(M))^{\tempere}_{\reel}\to\Irr(\cH(H))_t.\]
\end{lem}
\begin{proof}
Recall the group $\tH$ defined at the beginning of the proof of
Lemma~\ref{lem:inertias}: we have $H=\tH/Z$. 
Let $\cH(\tH)$ and $\cH(H)$ be the Iwahori-Hecke algebra of $\tH$ and $H$,
respectively. As in \cite[\S~1.5]{R}, we identify $\cH(H)$ with the set of
$Z$-fixed points of $\cH(\tH)$: 
\[\cH(H)=\cH(\tH)^Z.\]
The irreducible representations of $\cH(\tH)$ are parametrized by pairs
$(\Phi,\varrho)$ as above.

Let $\Irr(\cH(\tH))_t$ denote the set of (equivalence classes) of 
irreducible representations of $\cH(\tH)$ that are parametrized by 
pairs $(\Phi,\varrho)$ where $\Phi(\varpi_F,1)=t$. 
The construction of these representations is due Kazhdan and Lusztig \cite{KL}.
We will use the description given in \cite{CG}. We first recall
the standard modules 
\[M_{t,x,\varrho}^\tH:=\Hom_{\Cent_{\tH}(t,x)}(\varrho, \H_*(\bB^\Upsilon,\Cset)).\]
The similar construction of representations in $\Irr(\cH(\cM(t))_1$ leads
to the standard modules:
\[M_{1,x,\varrho}^{M^0}:=\Hom_{\Cent_{M^0}(x)}(\varrho,
\H_*(\bB^\Xi,\Cset)).\]
We have 
\[\Cent_{M^0}(x)=\Cent_{\tH}(t,x)\]
(\ie the centralizer in $M^0$ of $x$ coincides with the simultaneous 
centralizer in $\tH$ of $t$ and $x$), and (see the proof of \ref{lem:UP}):
\[
\H_*(\bB^{\Upsilon}, \Cset) = \H_*(\bB_1^{\Psi}, \Cset) \oplus \cdots
\oplus \H_*(\bB_m^{\Psi}, \Cset)
\]
is a sum of equivalent $A$-modules.
It follows that the irreducible representation of $\cH(\tH)$ attached to
$(t,x,\varrho)$ is \emph{isomorphic} to
the irreducible representation of $\cH(\cM(t))$ attached to
$(1,x,\varrho)$.

On the other hand, it follows from \cite[Theorem~A.13]{RamRam} (see also
\cite[Theorem~1.5.1]{R}) combined with \cite[(3.4.1)]{R} that the set of 
(equivalence classes) of
irreducible representations of $\cH(H)=\cH(\tH)^Z$ which are in 
$\Irr(\cH(\tH))_t$ are parametrized by
the triples $(x,\varrho,\psi)$ where $\psi$ runs over the simple modules of
the twisted group algebra
$\Cset[(A^+/A)_\varrho]_{\square[\varrho]}$.
\end{proof}
 
 \section{The Main Result} Let $\chi$ be a ramified character of the maximal 
torus $\cT \subset \cG$. We recall that
 \begin{align*}
 \fs & = [\cT,\chi]_{\cG}\\
 H & = \Cent_G(\im \hat{\chi})
 \end{align*}
 and $\Irr(\cH(H))_t$ is the set of (equivalence classes) of 
irreducible representations of the extended Iwahori-Hecke algebra 
$\cH(H)$ that are parametrized by pairs $(\Phi,\rho)$ where 
$\Phi(\varpi_F,1)=t$.   

We now assemble our results: 
 \begin{align}\label{fs0}
W^{\fs} & = \cW^H\\ \label{fs1}
(T\q \cW^{\fs})_2  & =  \{(t,\tau) : t \in T, \tau \in \Irr(\cW_t^H)\}/\cW^H\\ \label{fs2}
\cW^H_t & = \cW^{M^0} \rtimes \pi_0(M)\\ \label{fs23}
\Irr\, \cW^H_t  & =  (\Irr \, \cW^{M^0}\q \pi_0(M))_2^{\square}\\ \label{fs3}
\Irr(\Cset[\cW^{M^0}] \rtimes \pi_0(M)) & \simeq  \Irr(\cH(\cM(t)) \rtimes
\pi_0(M))^{\tempere}_{\reel}\end{align}
\begin{equation} 
\label{fs4} 
\Irr(\cH(\cM(t))\rtimes\pi_0(M))^{\tempere}_{\reel}  \simeq \Irr(\cH(H))_t
\end{equation}
The union of the sets $\Irr(\cH(H)_t$  
is in bijection with the set $\Irr(\cG)^{\fs}$, see \cite{Roc}:
\begin{align}\label{fs5}
\bigsqcup_{t\in T/\cW^H}\Irr(\cH(H))_t & \simeq \Irr(\cG)^{\fs}.
\end{align}

By combining Eqns.(\ref{fs0}, \ref{fs1}, \ref{fs2}, \ref{fs23}, \ref{fs3}, \ref{fs4}, \ref{fs5}), we create a  
canonical bijection between the extended quotient of the second kind $(T\q W^{\fs})_2$ and $\Irr(\cG)^{\fs}$:
 \begin{align}
\mu \colon (T\q W^{\fs})_2 \longrightarrow \Irr(\cG)^{\fs}, \quad \quad (t, x, \varrho, \psi) \mapsto (\Phi, \varrho \rtimes \psi).
\end{align}
Quite specifically, we have the following result.
\begin{thm} \label{thm:ps}  Let $\cG$ be a  split reductive $p$-adic group with connected centre and let $\Irr(\cG)^{\fs}$ be a Bernstein component in the principal series of $\cG$. Then there is a  canonical bijection 
\[
\mu^{\fs} : (T\q W^{\fs})_2 \to \Irr(\cG)^{\fs}
\]
where $\Irr(\cG)^{\fs}$ is parametrized by Kazhdan-Lusztig-Reeder parameters, and $(T\q W^{\fs})_2$ is the extended quotient of the second kind.
\end{thm}

This in turn creates a bijection 
 \begin{align}\label{bijection}
T\q W^{\fs} \longrightarrow \Irr(\cG)^{\fs}.
\end{align}
This bijection is not canonical in general, depending as it does
on a $c$-$\Irr$ system.
When $G = \GL(n)$, the finite group $\cW^H_t$  is a product of
symmetric groups: in this case there is a canonical  $c$-$\Irr$ system,    according to the classical theory of Young tableaux. 

\smallskip

\section{Correcting cocharacters} \label{sec:corcoc}
Let $\fs$ be a point in the Bernstein spectrum for the principal series of $\cG$. 
We will construct, for each $\alpha \in \Cset^{\times}$,  a  commutative diagram:
\begin{equation} \label{eqn:CD}
\begin{CD}
(T\q W^{\fs})_2 @> \mu^{\fs} >>  \Irr(\cG)^{\fs}\\
@ V \pi_{\alpha} VV        @VV i_{\alpha} V\\
T/W^{\fs} @=  T/W^{\fs}
\end{CD}
\end{equation}
The map $\mu^{\fs}$ is bijective and canonical, by Theorem (\ref{thm:ps}).   In this diagram, only the vertical maps depend on $\alpha$.
 
\medskip

We start by defining the vertical maps  in the diagram.  Let 
$\alpha \in \Cset^{\times}$.  Recall
the following matrix in $\SL(2,\Cset)$:
\[
Y_{\alpha} = 
\left(\begin{array}{cc}
\alpha & 0\\
0 & \alpha^{-1}\end{array}\right)
\]
and  set, as in Eqn.(\ref{hh}), 
\begin{equation} \label{eqn:defncochar}
h_{\gamma}\colon \Cset^{\times} \to T, \quad \quad \alpha \mapsto  \gamma(Y_{\alpha}).
\end{equation}
Then $h_{\gamma}$ is an element in the  \emph{coweight lattice} $X_{\bullet}(T)$,  also called  the \emph{cocharacter lattice}.
\medskip

Let $(\Phi,\rho) \in \Irr(\cG)^{\fs}$  be the Reeder parameter which corresponds, via $\mu^{\fs}$,  to the point $(t,\tau) \in (T\q W^{\fs})_2$.
We will define
\begin{align}
i_{\alpha}\colon \Irr(\cG)^{\fs}  \to T/W^{\fs}, \quad  (\Phi, \rho) \mapsto  \Phi (\varpi_F, Y_{\alpha})
\end{align}
\begin{align}
\pi_{\alpha} \colon (T\q W^{\fs})_2  \to T/W^{\fs}, \quad (t,\tau) \mapsto t \cdot \gamma (Y_{\alpha})
\end{align}
%$\mu^{\fs}(t,\tau) = (\Phi, \rho)$ is a Reeder parameter, and $(t,\tau) \in (T\q W^{\fs})_2$. 
We note that 
\[
 \Phi (\varpi_F, Y_{\alpha}) = t \cdot \gamma(Y_{\alpha})
\]
by the definition (\ref{Phi}) of the $L$-parameter $\Phi$, so that the diagram is commutative.

%ICI

%Consider the pair $(t,\tau)$ and let $\tau = \tau(x,\varrho)$.   Choose a unipotent class $[u]$ and consider the set of all triples $(t,x,\varrho)$
%for which $\exp x \in [u]$. 

\begin{prop} \label{prop:cocharcell}
The cocharacter $h_\gamma$ defined in Eqn.~{\rm (\ref{eqn:defncochar})}
depends only on the unipotent class in $H$ containing $\exp x$.
\end{prop}
\begin{proof} See Lemma \ref{lem:cocharindep}.
\end{proof}

\smallskip

In view of the Lusztig bijection between unipotent classes in $H$ and 
two-sided cells in $X(T)\rtimes W^\fs$,
we may set
\begin{equation} \label{eqn:hc}
h_{\bc} = h_{\gamma},
\end{equation}
where $\bc$ is the two-sided cell which corresponds to the unipotent class
in $H$ containing $\exp x$.

\smallskip

 For the \emph{restriction} of $\pi_{\alpha}$ to  $(T\q W^{\fs})^{\bc}_2$  we therefore have
\[
 (T\q W^{\fs})^{\bc}_2 \to T/W^{\fs}, \quad \quad (t,\tau) \mapsto  t\cdot h_{\bc}(\alpha)
\]
 We have created a system of correcting  cocharacters as defined in \cite[p.131]{ABP3}.
 \medskip

%$\bullet$ Let $s = 1$, and assume, for the moment, that $\Cent_H(t)$ is connected.    The map $\mu$ in Theorem 4.1  sends $(t, \tau) $ to $(\Phi ,\rho)$.   We %note that \[t = \Phi(\varpi_F, T_1) = \Phi(\varpi_F,1).\]The map $\mu$  determines the map \[(t, \tau) \mapsto (t, \Phi(1,u_0), \rho)\]
%which, in turn, determines the map \[\tau \mapsto (\exp(x), \rho) \] which is the Springer correspondence for the Weyl group $W_H(t)$.
The correspondence $\nu \mapsto \chi_{\nu}$ between points in $T$ and unramified quasicharacters of $\mathcal{T}$ can be fixed by the relation \[\chi_{\nu}(\lambda(\varpi_F)) = \lambda(\nu)\]where $\varpi_F$ is a uniformizer in $F$, and 
$\lambda \in X_{\bullet}(\mathcal{T}) = X(T)$.   

Let $X$ denote the rational cocharacter group of $\cT$,
identified with the rational character group of $T$ . Let $\cT_0$ be the maximal compact
subgroup of $\cT$ . A choice of uniformizer in $F$ gives a splitting $\cT = \cT_0 \times X$, so
characters of $\cT$ have the form $\chi \otimes \nu$, where $\chi$ is a character of $\cT_0$, and $\nu \in T$.  As in \S6, 
$\chi\otimes \nu$  gives rise to a homomorphism $\hat{\chi} : U_F \to T$.  This homomorphism is independent of $\nu$. \emph{The notation} $\hat{\chi}$ 
\emph{is consistent with the notation in} \S6.

Let  $q = q_F$ denote the cardinality of the residue field $k_F$ of $F$.  
\medskip

\textsc{Notation}.  Here is the dictionary for translation between Reeder's notation \cite[p.117]{R} and ours:   $s = t,\, \tau = \nu,\,\phi_u = \gamma, \,
\tau_u = \gamma(Y_{\sqrt q}), \, \tau = \tau_u \cdot s =   \gamma(Y_{\sqrt q})\cdot t, \, \hat{\chi}_s(u\varpi_F) = \hat{\chi}(u)\cdot t$.     
The definition of $\Phi$ in \cite[p.117]{R} therefore coincides with our definition (\ref{Phi}).
\medskip

Now consider the induced representation
\[
\Ind_{\mathcal{B}}^{\mathcal{G}}(\chi \otimes \nu),
\]
where $\cB\supset\cT$ is the standard Borel subgroup in $\cG$, and
$\chi\in\Irr(\cT)$.
 
Recall the \emph{infinitesimal character} \cite[\S2.1]{BDK}, also called the \emph{cuspidal support} map $\mathbf{Sc}$ \cite[VI.7.1.1]{Renard}:
\[
\mathbf{Sc}:  \Irr(\cG) \longrightarrow \Omega(\cG)
\]
where $\Omega(\cG)$ is the set of all $\cG$-conjugacy classes of cuspidal pairs.  

Let $\Irr_{\chi}(\cG, \cB)$  denote the set of irreducible representations of $\cG$, up to equivalence, which appear in $\Ind^{\cG}_{\cB}(\chi \otimes \nu)$, for some  $\nu \in T$, as in \cite[p.120]{R}.   
%Combining  we therefore have a bijection (5.3d)	i? : Irr?(G, B) ?? ??ö(G, B).
Then $\Irr_{\chi}(\cG, \cB)$ is in bijection with  the set of $G$-conjugacy classes
of pairs $(\Phi, \rho)$, where  $\Phi \colon W_F \times \SL(2, \Cset) \to G$ is a Langlands parameter such that
$\bB^{\Phi}$ is nonempty, $\Phi|_{U_F}$	is $G$-conjugate to $\hat{\chi}$, and $\rho \in \mathcal{R}(A_{\Phi}, \bB^{\Phi})$.  See \cite[4.5, 5.3d]{R}.
The cuspidal support of the irreducible representation of $\cG$ with Reeder parameter $(\Phi, \rho)$ is therefore the cuspidal pair $(\cT, \chi \otimes \nu)$. 

We fix $\cT$ and $\chi$ and allow $\nu$ to vary.   For the cuspidal pairs of the form $(\cT, \chi \otimes \nu)$, we have the identification, as in \cite[2.1]{BDK}:
\begin{align}\label{identify}
\{(\cT,\chi \otimes \nu) \colon \nu \in T\} \simeq T, \quad \quad (\cT,\chi \otimes \nu) \mapsto \nu.
\end{align}
This is how the quotient variety $T/W^{\fs}$ arises as a \emph{component} in the Bernstein variety $\Omega(\cG)$, see \cite[2.1]{BDK}.
 
Note that
\begin{equation} \label{eqn:iq}
i_{\sqrt q}\colon \Irr(\cG)^{\fs} \to T/W^{\fs}, \quad \quad (\Phi,\rho) \mapsto \Phi(\varpi_F, Y_{\sqrt q})
\end{equation}
and
\[
\Phi(\varpi_F, Y_{\sqrt q}) = t \cdot \gamma(Y_{\sqrt q}) = \nu 
\]
which is now identified by (\ref{identify}) with the cuspidal pair $(\cT, \chi \otimes \nu)$.  We infer that
\[
i_{\sqrt q}: \Irr(\cG)^{\fs} \to T/W^{\fs}, \quad \quad (\Phi,\rho) \mapsto (\cT, \chi \otimes \nu)
\]
so that 
\begin{align}
i_{\sqrt q} =\mathbf{Sc}.
\end{align}

This creates the commutative diagram
\[
\begin{CD}
(T\q W^{\fs})_2 @> \mu^{\fs} >>  \Irr(\cG)^{\fs}\\
@ V \pi_{\sqrt q} VV        @VV \mathbf{Sc} V\\
T/W^{\fs} @=  T/W^{\fs}
\end{CD}
\]
so that 
\begin{align}
\pi_{\sqrt q} = \mathbf{Sc} \circ \mu^{\fs}.
\end{align}
 This confirms our geometric conjecture, as stated in \cite{ABP3}, for the principal series of $\cG$.

Since $\mu^{\fs}$ is bijective and the diagram is commutative,  the number of points 
in the fibre of $\pi_{\sqrt q}$ equals the number of inequivalent irreducible constituents of
$\Ind_{\mathcal{B}}^{\mathcal{G}}(\chi \otimes \nu)$:
\begin{align}\label{card}
|\pi^{-1}_{\sqrt q}(\chi \otimes \nu)| = |\Ind_{\cB}^{\cG}(\chi \otimes \nu) |
\end{align}
The map $\pi_{\sqrt q}$, a finite morphism of algebraic varieties,  is therefore a model of the map  $\mathbf{Sc}$.   The predictive power of Eqn.(\ref{card}) is illustrated by  Example 2 in \S13. 

\section{Intersections of $L$-packets with $\Irr(\cG)^\fs$} \label{sec:Lpackets}
In \cite[\S10]{ABP4}, we raise the following question: ``In the ABP view of the smooth dual $\Irr(\cG)$, 
what are the $L$-packets?''

%In this section, we answer  the question ``In the ABP view of 
%$\widehat{G}$, what are the $L$-packets?'' in the principal series case. 

In the conjectural answer \cite[\S~10]{ABP4}, the essential point was 
that in the list of correcting
cocharacters $h_1$, $h_2$, $\ldots$, $h_r$, there may
be repetitions, i.e. it may happen that for some $i, j$ with
$1\le i < j \le r$, one has $h_i=h_j$, and these repetitions give rise to
$L$-packets. We find here that this is close to the truth: for a given
$L$-packet all the $h_i$ have to be the same, however one cocharacter
 $h_i$ may correspond to two distinct $L$-packets. The correct answer
in the principal series case is given by the following theorem.

\smallskip

We recall the cell decomposition
\[(T\q W^\fs)_2=\bigsqcup (T\q W^\fs)_2^\bc.
\]

\begin{thm} \label{thm:Lpackets}  
Two points $(t, \tau)$ and $(t^{\prime}, \tau^{\prime})$ in
$(T\q W^\fs)_2$ have
\[
\text{$\mu^\fs(t, \tau)$ and $\mu^\fs(t',\tau^{\prime})$ in the
same $L$-packet} 
\]
if and only
\[\text{$(t, \tau)$ and $(t^{\prime}, \tau^{\prime})$ belong to the same
$(T\q W^\fs)_2^\bc$}\]
  and
\[
\text{$\pi_{\alpha}(t,\tau) =
\pi_{\alpha}(t^{\prime},\tau^{\prime})$, for each $\alpha\in\Cset^\times$.}\]
\end{thm}
\begin{proof}
Let $(t, \tau)$ and $(t^{\prime}, \tau^{\prime})$ two points in
$(T\q W^\fs)_2$. Their images by the bijection $\mu^\fs$ are in the same
$L$-packet if and only if they correspond to the same of 
$L$-parameter.

Recall the definition of the $L$-parameter $\Phi$ (see Eqn.~(\ref{eqn:Phi})): 
\[\Phi \colon F^{\times} \times \SL(2,\Cset) \to G, \quad\quad
(u\varpi_F^n,Y) 
\mapsto \hchi(u)\cdot t^n\cdot \gamma(Y),\]
for all  $u \in U_F$, $n \in\Zset$, $Y \in \SL(2,\Cset)$.

Here $\hat\chi$ is a homomorphism from $U_F$ to $T$ provided by 
$\chi\otimes\nu$, and, as we have already observed, it is independent of
$\nu\in T$. In other words, $\hat\chi$ only depends on $\fs=[\cT,\chi]_\cG$. 
Since the inertial pair $\fs$ is fixed, it follows that the restriction of 
$\Phi$ to ${U_F}$ is fixed too.  

Now we have seen in the proof of Proposition~\ref{prop:cocharcell} that 
$\gamma$ depends on a unipotent class $\cU$ in $H$. Let $\bc$ be the
two-sided cell in $X(T)\rtimes W^\fs$ which corresponds to $\cU$  
by the Lusztig bijection.
 
Set \[(\Phi,\rho):=\mu^\fs(t,\tau)\quad\text{and}\quad
(\Phi',\rho'):=\mu^\fs(t',\tau').\]
Then $(t,\tau)$ and $(t',\tau')$ are in the same cell $(T\q W^\fs)_2^\bc$
if and only if $\Phi$ and $\Phi'$ have same restriction to $\SL(2,\Cset)$.
 
Finally, from the commutativity of the diagram~(\ref{eqn:CD}), we have
\[\Phi(\varpi_F,Y_{\alpha})=t\cdot \gamma(Y_{\alpha})=t\cdot
h_\gamma(Y_{\alpha})=\pi_{\alpha}(t,\tau),\]
and similarly
\[\Phi'(\varpi_F,Y_{\alpha})=\pi_{\alpha}(t',\tau').\]
The result follows.

Finally, recall the definition of the correcting cocharacter attached to $\bc$, from
Eqns.~(\ref{eqn:defncochar}) and (\ref{eqn:hc}):
\[h_{\bc}\colon \Cset^{\times} \to T, \quad \quad \alpha \mapsto
\gamma(Y_{\alpha}).\]
Hence all the $\mu^\fs(t,\tau)$ in a given $L$-packet correspond to the
same cocharacter. 
\end{proof}

 \section{Examples}

%\noindent
\textsc{Example~1.}  \emph{Realization of the ordinary quotient}  $T/W^\fs$.    
Let $\fP(G)$ denote the set of $G$-conjugacy classes
of Langlands parameters \[\Phi \colon W_F \times \SL(2, \Cset) \to
G\] such that
$\bB^{\Phi}$ is nonempty and $\Phi|_{U_F}$ is $G$-conjugate to $\hat{\chi}$.
Consider an $L$-parameter $\Phi\in \fP(G)$ for which 
$\Phi | _{\SL(2,\Cset)} = 1$ (hence $\Phi \colon W_F \to G$).
Let $t = \Phi(\varpi_F)$. 

Now $t$ is a semisimple element in $G$, and all such semisimple elements arise.  
Modulo conjugacy in $G$, the set of such $L$-parameters $\Phi$ is  
parametrized by the quotient $T/W^\fs$.  Explicitly, let
\[
\fP_1(G): = \{\Phi\in\fP(G) \,:\,   \Phi |_{ \SL(2,\Cset)} = 1\}.
\]  
   Then we have a canonical bijection
\[
\fP_1(G) \to T/W^\fs, \quad \quad \Phi \mapsto
\Phi(\varpi_F,1)
\]
which fits into the commutative diagram
\[
\begin{CD}
\fP_1(G) @>>> T/W^\fs \\
@VVV              @VVV   \\
\fP(G)  @>>> (T \q W^\fs)_2
\end{CD}
\]
where the vertical maps are inclusions. 

\smallskip

From the construction of the system of correcting cocharacters, it follows
that $\fP_1(G)$ is attached to exactly one cocharacter $h_\bc$, the one
which is indexed by the two-sided cell $\bc$ corresponding to the trivial
unipotent class in $H$. Hence $\bc$ equals $\bc_0$: the \emph{lowest} 
two-sided cell in $X(T)\rtimes W^\fs$. 

Recall the corresponding part $(T\q W^{\fs})_2^{\bc_0}$ (see 
Eqn.~(\ref{eqn:T2cbis})) of the extended quotient of the second kind. 
Then we get
\begin{equation}
\label{eqn:TW2c0}
T/W^\fs\subset (T\q W^{\fs})_2^{\bc_0}.\end{equation}
By using the commutative diagram~(\ref{eqn:CDnu}), we obtain,
as conjectured in \cite[p.~87]{ABP4}:
\begin{equation}
\label{eqn:TWc0}
T/W^\fs\subset (T\q W^{\fs})^{\bc_0},\end{equation}
where $(T\q W^{\fs})^{\bc_0}$ is the part corresponding to the lowest 
two-sided cell $\bc_0$ in the cell decomposition~(\ref{eqn:cellddec}) of 
the extended quotient $T\q W^{\fs}$.

\smallskip

On the other hand, the corresponding group $\Cent_{M^0}(x)=M^0$ (here $x=1$)
is connected and acts trivially in homology. Therefore $\varrho$ is the unit
representation $1$. 
Then Eqn.~(\ref{eqn:T2c}) gives:
\begin{equation}
\widetilde T_2^{\bc_0}=\{(t,\psi)\in \widetilde
T_2\,:\,\psi\in\pi_0(M)\}.\end{equation}

\bigskip

%\noindent
\textsc{Example $2$.}   An example in the principal series of $\SL(4,\Qset_2)$.   Note that, for this group, the centre is not connected. The multiplicative group of the field $\Qset_2$ is given by
\[
\Qset_2^{\times} \simeq \Zset \times \Zset_2 \times \Zset/2\Zset
\]
and so $\Qset_2^{\times}$ admits three ramified characters of order $2$. They may be written $\eta,\, \chi,\, \eta \cdot \chi$.
\smallskip

The three ramified quadratic characters of $\Qset_2^{\times}$ create a unitary character, of order $2$,  of the standard Borel subgroup $\cB \subset \SL(4, \Qset_2)$:
\[
\tau : \left[
\begin{array}{cccc}
x_1 & * & * & *\\
0 & x_2 & * & *\\
0 & 0 & x_3 & *\\
0 & 0 & 0 & x_4
\end{array}
\right] \mapsto \eta(x_2)\chi(x_3)(\eta \cdot \chi)(x_4)
\]
 
  We twist the character $\tau$ by an unramified unitary character  $\psi$  and form the induced representation
$\Ind_{\cB}^{\cG}(\psi \tau)$.    Let $\Psi^1(\mathcal{T})$ denote the group of unramified unitary characters of the maximal torus $\cT \subset \SL(4)$, and let
\[
E^{\fs}: = \{\psi \otimes \tau : \psi \in \Psi^1(\cT)\}.
\]
Then $E^{\fs}$ has the structure of a compact torus.   The subgroup of the Weyl group which fixes $E^{\fs}$ is $W^{\fs}: = \bZ/2\bZ \times \bZ/2\bZ$.   We have the standard projection
\[
\pi^{\fs}: E^{\fs}\q W^{\fs} \to E^{\fs}/W^{\fs}\]
of the extended quotient onto the ordinary quotient.   To simplify notation, we will sometimes write $\pi = \pi^{\fs}$. 

The extended quotient $E^{\fs}\q W^{\fs}$ is the disjoint union of $6$ unit intervals $a,b,c,d,e,f$  and the ordinary quotient $E^{\fs}/W^{\fs}$. In the projection $\pi$, these $6$ intervals assemble themselves into the $6$ edges of a tetrahedron in $E^{\fs}/W^{\fs}$.  The cardinality of each fibre of $\pi$ creates a perfect model of reducibility.   The locus of reducibility is the $1$-skeleton $\mathfrak{R}$ of a tetrahedron, and we have
\[
|\pi^{-1}(\psi\tau)| = |\Ind_{\cB}^{\cG}(\psi\tau)|
\]
for all unramified unitary characters $\psi$ of $T$.   On the interior of each edge $\pi(a), \ldots, \pi(f)$ of $\mathfrak{R}$,  each induced representation admits
 $2$ distinct irreducible constituents; on each vertex of $\mathfrak{R}$, each induced representation admits $4$ distinct irreducible components.
 
 The detailed computations are in \cite{CP}; the theory of the $R$-group for $\SL(N)$, on which these computations depend, is to be found in Goldberg \cite{G}. 

The pre-image of the interior of one edge is the union of two open intervals (the one corresponding to the given edge and one in the ordinary quotient), replicating the fact that the $R$-group has order $2$, while the pre-image of a vertex is the union of three endpoints of intervals and one point in the ordinary quotient, replicating the fact that the $R$-group has order $4$ here. The $1$-skeleton of the tetrahedron is a perfect model of reducibility and confirms the geometric conjecture in this case.   Quite specifically, let
 $\cG = \SL(4, \Qset_2)$ and $\fs = [\cT,\tau]_{\cG}$.   There exists a continuous bijection
\[
\mu^{\fs} : E^{\fs}\q W^{\fs} \to \Irrt(\cG)^{\fs}
\]
such that
\begin{align}
\Sc \circ \mu^{\fs} = \pi^{\fs}.
\end{align}
 
\textsc{Example~3}. \emph{The general linear group}.  Let $\cG = \GL(n), G = \GL(n,\Cset)$. Let
\[
\Phi = \chi \otimes \tau(n)
\]
where $\chi$ is an unramified quasicharacter of $\cW_F$ and $\tau(n)$ is the irreducible $n$-dimensional representation of $\SL(2,\Cset)$.    By local classfield theory, the quasicharacter $\chi$ factors through $F^{\times}$.   In the local Langlands 
correspondence for $\GL(n)$, the image of $\Phi$ is the unramified twist $\chi  \circ  \det$ of the Steinberg representation $\St(n)$.  

The sign representation $sgn$ of the Weyl group $W$ has Springer parameters $(\mathcal{O}_{prin},1)$, where $\mathcal{O}_{prin}$ is the principal orbit in $\mathfrak{gl}(n,\Cset)$. In the \emph{canonical} correspondence between irreducible representations of $S_n$ and conjugacy classes in $S_n$, the trivial representation of
$W$ corresponds to the  conjugacy class containing  the $n$-cycle $w_0 = (123 \cdots n)$.

Now $G_{\Phi} = C(im \, \Phi)$ is connected \cite[\S3.6.3]{CG}, and so acts trivially in homology. 
   Therefore $\rho$ is the unit representation $1$.  The image $\Phi(1,u_0)$ is a regular nilpotent, i.e. a nilpotent with one Jordan block (given by the partition of $n$ with one part).   The corresponding conjugacy class in $W$ is $\{w_0\}$.   The corresponding irreducible component  of the extended quotient is
$$T^{w_0}/Z(w_0) = \{(z,z, \ldots,z): z \in \Cset^{\times}\}  \simeq \Cset^{\times}.$$  This is our model, in the extended quotient picture, of the complex $1$-torus of all unramified twists of the Steinberg representation $\St(n)$. The map from $L$-parameters to pairs $(w,t) \in T \q W$ is given by
\[
\chi \otimes \tau(n) \mapsto (w_0, \chi(\Frob), \dots, \chi(\Frob)).
\]
Among these representations, there is one real tempered representation, namely $\St(n)$, with $L$-parameter $1 \otimes \tau(n)$, 
attached to the principal orbit $\cO_{prin} \subset G$.

More generally, let 
\[
\Phi = \chi_1 \otimes \tau(n_1) \oplus \cdots \oplus \chi_k \otimes \tau(n_k)
\]
where $n_1 + \cdots + n_k = n$ is a partition of $n$.   This determines the unipotent orbit $\cO(n_1, \ldots, n_k) \subset G$.  There is a  conjugacy class
in $W$ attached canonically to this orbit: it contains the product of disjoint cycles of lengths $n_1, \ldots, n_k$. The fixed set is a complex torus, and the
component in $T\q W$ is a product of symmetric products of complex $1$- tori.  See \cite{BP} for more details.
\smallskip

Denote by $\fii$ the Iwahori point in the Bernstein spectrum of $\cG(F): = \GL(n,F)$, so that $\fii = [\cT,1]_{\cG}$. 
Let $T$  be the standard maximal torus in $\GL(n,\Cset)$.   The Weyl group is the symmetric group $S_n$.   We will denote our bijection, in this case canonical, as follows:
\[
\mu_{F}^{\fii} : T\q W \to \Irr(\cG(F))^{\fii}
\]
Let $E/F$ be a finite Galois extension of the local field $F$.  According to \cite[Theorem 4.3]{MP}, we have a commutative diagram
\[
\begin{CD}
T\q W @> \mu_{F}^{\fii} >>  \Irr(\cG(F))^{\fii}\\
@ V   VV        @VV \RBC_{E/F} V\\
T\q W @> \mu_{E}^{\fii} >>  \Irr(\cG(E))^{\fii}
\end{CD}
\]
In this diagram, the right vertical map $\RBC_{E/F}$ is the standard base change map  sending one Reeder parameter to another as follows:
\[
(\Phi,1) \mapsto (\Phi_{|W_E},1).
\]

Let  \[f = f(E,F)\] denote the residue degree of the extension $E/F$.    We proceed to describe the left vertical map.   We note that  the action of W on T is as automorphisms of the algebraic group T.   Since  $T$  is a group, the map \[T \to T, \quad  t \mapsto t^f\]  is well-defined for any positive integer $f$.   The map
\[
\widetilde{T} \to \widetilde{T}, \quad (t,w) \mapsto (t^f,w)
\]
is also well-defined, since
\[
w\cdot t^f = wt^fw^{-1} = wtw^{-1}wtw^{-1} \cdots wtw^{-1} = t^f.
\]
Since
\[
\alpha\cdot(t^f) = (\alpha\cdot t)^f
\]
for all $\alpha \in W$, this induces a map
\[
T\q W \to T\q W
\]
which is an endomorphism (as algebraic variety) of the extended quotient $T\q W$.  We shall refer to this endomorphism as the \emph{base change endomorphism of degree $f$.}  The left vertical map is the base change endomorphism of degree $f$, according to \cite[Theorem 4.3]{MP}.  That is, our bijection $\mu^{\fii}$  is compatible with base change for $\GL(n)$. 

When we restrict our base change endomorphism from the extended quotient $T\q W$ to the ordinary quotient $T/W$, we see that the  commutative diagram 
containing $\RBC_{E/F}$ is consistent with \cite[Lemma 4.2.1]{Haines}.

\bigskip

\emph{Acknowledgement}. We would like to thank  A. Premet for drawing our attention to reference \cite{CG}.


\begin{thebibliography}{99}
%\bibitem{A} J.F. Adams, Lectures on Lie groups, Benjamin, New York 1969.
\bibitem{ABP1} A.-M. Aubert, P. Baum, R.J. Plymen, The Hecke algebra of a reductive p-adic group: a geometric conjecture.  Aspects of Mathematics {\bf 37}, 
Vieweg Verlag (2006) 1--34.
\bibitem{ABP2} A.-M. Aubert, P. Baum, R.J. Plymen, Geometric structure in the representation theory of
$p$-adic groups,  C.R. Acad. Sci. Paris, Ser. I {\bf 345} (2007)
573--578.
\bibitem{ABP3} A-M. Aubert, P. Baum, R.J. Plymen,  Geometric structure in the principal series of the $p$-adic group $G_2$,  
Represent. Theory {\bf 15} (2011) 126--169.
\bibitem{ABP4} A.-M. Aubert, P. Baum, R.J. Plymen, Geometric structure in the representation theory of $p$-adic groups II, Contemporary Math.
{\bf 534} (2011) 71--90. 
\bibitem{Bara} V. Baranovsky, Orbifold cohomology as periodic cyclic homology, 
Int. J.~Math. 14 (2003) 791--812.
\bibitem{BM} D. Barbasch, A. Moy, A unitarity criterion for $p$-adic
groups, Invent. Math. {\bf 98} (1989) 19--37.
%\bibitem{Carter} R.~Carter, Finite Groups of Lie Type, Conjugacy classes
%and complex characters, Wiley Classics Library, 1993.
\bibitem{BDK} J. Bernstein, P. Deligne, D. Kazhdan, Trace Paley-Wiener theorem for reductive $p$-adic groups, J. Analyse Math.
{\bf 47} (1986) 180--192. 
\bibitem{BP} J. Brodzki, R.J. Plymen, Complex structure on the smooth dual of $\GL(n)$, Documenta Math. {\bf 7} (2002) 91--112.
%\bibitem{B} J.L. Brylinski, Cyclic homology and equivariant theories, Ann. Inst. Fourier {\bf 37} (1987) 15--28. 
\bibitem{B} D. Burghelea, Cyclic homology of the group rings, Comment. Math. 
Helvetici 60 (1985) 354--365.
\bibitem{CP} K.F. Chao and R.J. Plymen, Geometric structure in the tempered dual of $\SL(4)$, \emph{Bull. London Math. Soc.}, to appear.
\bibitem{CG} N. Chriss and V. Ginzburg, Representation theory and complex
geometry, Birkhauser 2000.  
%\bibitem{Haines} T.J. Haines, Base change for Bernstein centers of depth zero principal series blocks,  arXiv:1012.4968[mathRT].
%\bibitem{CM} D.H. Collingwood and W.M. McGovern, Nilpotent orbits in
%complex semisimple Lie algebras, Reinhold, Van Nostrand, New York, 1993.
%\bibitem{D} D.I.~Deriziotis, Centralizers of semisimple elements in a
%Chevalley group, Comm. in Algebra, {\bf 9} (1981), pp.~1997--2014.
\bibitem{FR} R.~Fowler, G.~R\"ohrle, On cocharacters associated to
nilpotent elements of reductive groups, Nagoya Math.~J. {\bf 190} (2008),
105--128.
\bibitem{G} D. Goldberg,  $R$-groups and elliptic representations for $\SL_n$, \emph{Pacific J.
Math.} 165 (1994) 77-92.
\bibitem{Haines} T.J. Haines, Base change for Bernstein centers of depth zero principal series blocks,  arXiv:1012.4968[mathRT].
\bibitem{Hot} R.~Hotta, On Springer's representations, J. Fac. Sci. Uni.
Tokyo, IA {\bf 28} (1982) 863--876.
\bibitem{J} J.C.~Jantzen, Nilpotent Orbits in Representation Theory, in: Lie
Theory, Lie Algebras and Representations (J.-P.~Anker and B.~Orsted,
eds.), Progress in Math. {\bf 228}, Birkh\"auser, Boston 2004.
\bibitem{KL} D.~Kazhdan and G.~Lusztig, Proof of the Deligne-Langlands
conjecture for Hecke algebras, Invent. math. {\bf 87} (1987) 153--215.
%\bibitem{K} M. Khalkhali, Basic noncommutative geometry, EMS Series of Lectures in Math.,  2009.
\bibitem{LCellsIV} G. Lusztig, Cells in affine Weyl groups, IV, J. Fac.
Sci. Univ. Tokyo, Sect. IA, Math, {\bf 36} (1989) 297--328.
\bibitem{LuAst} G. Lusztig, Representations of affine Hecke algebras,
Ast\'erisque {\bf 171--172} (1989) 73--84.
%\bibitem{LuSpring} G.~Lusztig, Green polynomials and singularities of nilpotent classes, Adv. in Math. {\bf 42} (1981), 169--178.
\bibitem{Lu} G.~Lusztig, Cuspidal local systems and graded Hecke algebras, II, Canadian Math. Soc., Conference Proceedings, {\bf 16} (1995) 217--275.
\bibitem{McN} G.~McNinch, Optimal SL(2) homomorphisms	Comment. Math.
Helv., 2005 (80), 391-426.
\bibitem{MP} S. Mendes and R.J. Plymen, Base change and $K$-theory for $\GL(n)$, J. Noncommut. Geom. 1 (2007) 311 -- 331.
%\bibitem{M} J. Morava, HKR characters and higher twisted sectors,  Contemp. Math. 403 (2006) 143 --152.
\bibitem{RamRam} A. Ram and J.~Ramagge, Affine Hecke algebras, cyclotomic
Hecke algebras and Clifford theory, Birkh\"auser, Trends in Math. (2003),
428--466.
\bibitem{Renard} D. Renard, Repr\'esentations des groupes r\'eductifs $p$-adiques,  Cours Sp\'ecialis\'es 17, Soci\'et\'e Math. de France  2010. 
\bibitem{R} M.~Reeder, Isogenies of Hecke algebras and a Langlands correspondence
for ramified principal series representations,  Representation Theory {\bf
6} (2002) 101--126.
\bibitem{Roc} A. Roche, Types and Hecke algebras for principal
series representations of split reductive $p$-adic groups, Ann.
scient. \'Ec. Norm. Sup. {\bf 31} (1998) 361--413.
%\bibitem{Ruan} Y. Ruan, Stringy orbifolds, math.AG/0201123.
%\bibitem{Sh} F. Shahidi, Eisenstein series and automorphic $L$-functions, Colloquium Publication 58 , AMS 2010.
\bibitem{Sol} M.~Solleveld, On the classification of irreducible
representations of affine Hecke algebras with unequal parameters,
arXiv:1008.0177, 2010 (to appear in Represent. Theory). 
\bibitem{Shoji} T.~Shoji, Green functions of reductive groups over a
finite field", PSPM {\bf 47} (1987), Amer. Math. Soc. 289--301.
\bibitem{SpringerSteinberg} T.~Springer and R.~Steinberg, Conjugacy
classes, In Lecture Notes in Math. {\bf 131}, Springer-Verlag, Berlin,
1970, 167--266.
\bibitem{Steinberg} 
R.~Steinberg, Endomorphisms of linear algebraic groups, Mem. Amer. Math.
Soc. {\bf 80}, 1968. 
\bibitem{Steinbergtorsion}
R.~Steinberg, Torsion in reductive groups, Adv. Math. {\bf 15} (1975) 63--92. 
\end{thebibliography}
\end{document}